\documentclass{amsart}

\usepackage{url}
\usepackage{enumerate}
\usepackage{amssymb}

\usepackage[vlined,boxed,norelsize]{algorithm2e}
\usepackage{tabularx}

\makeatletter
\renewcommand{\@algocf@capt@boxed}{above}% formerly {bottom}
\makeatother

\usepackage{xcolor}

\usepackage{tikz}
\usepackage{pgfplots}

\newtheorem{theorem}{Theorem}
\newtheorem{proposition}[theorem]{Proposition}
\newtheorem{lemma}[theorem]{Lemma}
\newtheorem{corollary}[theorem]{Corollary}
\theoremstyle{definition}
\newtheorem{example}{Example}

\title{Parametrizing Arf numerical semigroups}

\author[Garc\'{\i}a-S\'anchez]{P. A. Garc\'{\i}a-S\'anchez}
\address{IEMath-GR and Departamento de \'Algebra, Universidad de Granada, E-18071 Granada, Espa\~na}
\email{pedro@ugr.es} %\url{www.ugr.es/local/pedro}

\author[Heredia]{B. A. Heredia}
\address{Departamento de Matem\'atica e Centro de Matem\'atica e Aplica\c{c}oes (CMA), FCT, Universidade Nova de Lisboa}
\email{b.heredia@fct.unl.pt}
\thanks{The second author is supported by the Funda\c{c}\~ao para a Ci\^encia e a Tecnologia (Portuguese Foundation for Science and Technology) through the project UID/MAT/00297/2013 (Centro de Matem\'atica e Aplica\c{c}\~oes).}

\author[Karaka\c{s}]{H. \.{I}. Karaka\c{s}}
\address{Faculty of Commercial Sciences, Ba\c{s}kent University, Ankara, Turkey}
\email{karakas@baskent.edu.tr}

\author[Rosales]{J. C. Rosales}
\address{Departamento de \'Algebra, Universidad de Granada, E-18071 Granada, Espa\~na}
\email{jrosales@ugr.es} %\url{www.ugr.es/local/pedro}

\thanks{The authors are supported by the projects MTM2014-55367-P, FQM-343,  FQM-5849, and FEDER funds}


\begin{document}

\begin{abstract}
We present procedures to calculate the set of Arf numerical semigroups with given genus, given conductor and given genus and conductor. We characterize the Kunz coordinates of an Arf numerical semigroup. We also describe Arf numerical semigroups with fixed Frobenius number and multiplicity up to six.
\end{abstract}

\maketitle

\section*{Introduction}

Du Val showed in \cite{Du-Val} how multiplicity sequences of the successive blow-ups of a curve can be used to classify singularities. His approach was geometric in nature and asked, while he was presenting his results in the University of Istambul, if there were an algebraic counterpart of his findings. Arf, who was attending Du Val's lecture said that the computation of Du Val's characters could be calculated by algebraic means, and after a week he showed how to do this. The results were published in \cite{arf} and later these characters were called Arf Characters of a curve.  

Arf's idea was to calculate what Limpan called later in \cite{lipman} the Arf ring closure of the coordinate ring of the curve, and then its value semigroup (which is an Arf numerical semigroup). The minimal generators of this semigroup are the Arf characters. 

The idea of using valuations and numerical semigroups to study curves was not new (this was already carried out by Zariski for plane algebraic curves \cite{zariski}), nor the idea of producing successive blowing ups of the curve (in three dimensions, this was done already by Semple in \cite{semple}).  Sert\"oz in \cite{sertoz} presents a historic overview and motivation of the problem (see also his appendix to \cite{arf-cw}).  

Arf numerical semigroups are always of maximal embedding dimension (see \cite{b-d-f} for this and other maximal properties on numerical semigroups and one dimensional local rings). This family of numerical semigroups has several nice properties; we summarize some of them. They are closed under finite intersections and if we adjoin to an Arf numerical semigroup its Frobenius number, then the resulting semigroup is again an Arf semigroup (\cite{arf-num-sem}; in other words, the set of Arf numerical semigroups is a Frobenius variety, see \cite{houston}). Arf numerical semigroups can be also defined by the patterns \cite{maria1, patterns} $x+y-z$ or $2x-y$: for all $x$, $y$ and $z$ nonnegative integers in the semigroup, if $x\ge y\ge z$, then $x+y-z$ is again in the semigroup. Moreover quotients (or fractions) by positive integers of Arf numerical semigroup are again Arf, see \cite{d-s}. Also parameters of algebro-geometric codes associated to these semigroups are well understood \cite{arf-codes}. Arf semigroups are acute semigroups, that is,  the last interval of gaps before the conductor is smaller than the previous interval of gaps \cite{maria-acute}.

In this manuscript we describe a way to calculate all Arf numerical semigroups with a prescribed genus and/or conductor. This is accomplished by means of Arf sequences, associating to each of these sequences an Arf numerical semigroup. We also characterize the Kunz coordinates of an Arf numerical semigroup. 

In the last section, with the use of Ap\'ery sets, we show how to parametrically describe all Arf numerical semigroups with fixed Frobenius number and multiplicity up to six.

The algorithms presented have been implemented in \texttt{GAP} \cite{GAP} and will appear in a forthcoming version of the accepted  \texttt{GAP} package \texttt{numericalsgps} \cite{numericalsgps}. The development version of \texttt{numericalsgps} is freely available at \url{https://bitbucket.org/gap-system/numericalsgps}. The reader interested in the implementation may have a look at the manual and the file \texttt{arf-med.gi} in the \texttt{gap} folder of the package.

\section{Notation}

We will follow the notation of \cite{ns}. The reader interested in plane curves and numerical semigroups can have a look at \cite[Chapter 4]{ns-app}. A nice description of one dimensional analytically irreducible local rings and their value semigroups can be found in \cite{b-d-f} (we also recommend this manuscript a good explanation on how the terminology used in numerical semigroups comes from Algebraic Geometry).

A \emph{numerical semigroup} $S$ is a submonoid of $\mathbb N$, the set of nonnegative integers, under addition and with finite complement in $\mathbb N$. A nonnegative integer $g$ not in $S$ is known as a \emph{gap} of $S$, and the cardinality of the set of gaps of $S$, $\mathbb N\setminus S$, is the \emph{genus} of $S$ (or degree of singularity of $S$, \cite{b-d-f}), denoted $\mathrm g(S)$. 

As $\mathbb N\setminus S$ has finitely many elements, the set $\mathbb Z\setminus S$ (with $\mathbb Z$ the set of integers) has a maximum, which is known as the \emph{Frobenius number} of $S$, $\mathrm F(S)$. In fact, the \emph{conductor} of $S$, denoted here by $\mathrm c(S)$, is the Frobenius number of $S$ plus one (\cite{b-d-f} explains the relationship with the conductor of the semigroup ring associated to $S$). 

For a nonempty set of nonnegative integers $A$ we denote by 
\[
\langle A\rangle =\left\{ \sum\nolimits_{a\in A} \lambda_a a\mid \lambda_a\in \mathbb N \hbox{ for all } a\in A\right\},
\]
the submonoid of $\mathbb N$ generated by $A$, where the sums have all but finitely many $\lambda_a$ equal to zero. We say that $A$ \emph{generates} $S$ if $\langle A\rangle =S$, and that $A$ is a \emph{minimal generating system} of $S$ if $A$ generates $S$ and no proper subset of $A$ has this property. Every numerical semigroup $S$ has a unique minimal generating system: $S^*\setminus (S^*+S^*)$, where $S^* =S\setminus\{0\}$ \cite[Chapter 1]{ns}. This minimal generating system must contain the \emph{multiplicity} of $S$, denoted $\mathrm m(S)$, which is the least positive integer in $S$.  The cardinality of $S^*\setminus (S^*+S^*)$ is the \emph{embedding dimension} of $S$. 

Since two minimal generators cannot be congruent modulo the multiplicity of $S$, it 
follows that the embedding dimension of $S$ is less than or equal to its multiplicity. Numerical semigroups attaining this upper bound are called \emph{maximal embedding dimension numerical semigroups}. There are many characterizations of the maximal embedding dimension property. One of them is the following (see for instance \cite{b-d-f} or \cite[Chapter 2]{ns}): a numerical semigroup has maximal embedding dimension if and only if for every $x,y\in S\setminus\{0\}$, the integer $x+y-\mathrm m(S)$ is in $S$.   

We are interested in this manuscript in a subfamily of maximal embedding numerical semigroups, which is the set of Arf numerical semigroups. A numerical semigroup has the \emph{Arf property} if for every $x,y,z\in S$ with $x\ge y\ge z$, we have $x+y-z\in S$ (from this definition it follows easily that Arf numerical semigroups have maximal embedding dimension). The Arf property on $S$ is equivalent to: $2x-y\in S$ for every $x,y\in S$ with $x\ge y$. 

Let $I\subseteq \mathbb Z$. We say that $I$ is a relative \emph{ideal} of $S$ if $I+S\subseteq I$ and there exists an integer $i$ such that $i+I\subseteq S$. Given $I$ and $J$ ideals of $S$, the set 
\[
I-_\mathbb Z J=\{ z\in \mathbb Z\mid z+J\subseteq I\}
\]
is again an ideal of $S$, as it is $nI=\{i_1+\cdots+i_n\mid i_1,\ldots, i_n\in I\}$ \cite{b-d-f}. The \emph{Lipman semigroup} of $S$ with respect to $I$ is defined as 
\[
\mathrm L(S,I)=\bigcup_{n\in\mathbb N} (nI-_\mathbb Z nI),
\]
and it is also called the \emph{semigroup obtained from $S$ by blowing-up $I$} \cite[Section I.2]{b-d-f}. 

An ideal $I$ is \emph{proper} if $I\subseteq S$. There is only a maximal proper ideal of $S$ with respect to set inclusion, and this ideal is precisely $\mathrm M(S)=S^*$ (so numerical semigroups are ``local"). We will refer to $\mathrm L(S)=\mathrm L(S,S^*)$ as the \emph{Lipman semigroup} of $S$. It can be shown (see for instance \cite[I.2.4]{b-d-f}) that if $\{n_1,\ldots,n_e\}=S^*\setminus (S^*+S^*)$ is the minimal generating set of $S$ with $n_1<\cdots < n_e$, then 
\[
\mathrm L(S)= \langle n_1,n_2-n_1,\ldots, n_e-n_1\rangle.
\]

\begin{example}\label{ex-3-5-7}
Let $S=\langle 3,5,7\rangle = \{0,3,5,\to \}$ (here $\to$ denotes that all integers larger than $5$ are in the semigroup; we are denoting in this way that the conductor of $S$ is $5$). Then $\mathrm L(S)=\langle 2,3\rangle$ and $\mathrm L(\mathrm L(S))=\mathbb N$. We obtain in this way a multiplicity sequence $3,2,1$ of the successive blowing-ups with respect to the maximal ideal.

Observe that in this setting $S=\{0, 3, 3+2, 3+2+1,\to\}$. 

If we repeat this calculations with $T=\langle 3,5\rangle$, then we have again  $\mathrm L(S)=\langle 2,3\rangle$ and $\mathrm L(\mathrm L(S))=\mathbb N$; whence the multiplicity sequence here is the same. However $T$ is not the semigroup ``spanned'' by this multiplicity sequence, which in this case is $S$.
\end{example}

The property that pops up in the above example is not accidental. Indeed by \cite[Theorem I.3.4]{b-d-f}, a numerical semigroup $S$ has the Arf property if and only if 
\begin{equation}\label{arf-mult-seq}
S=\left\{0, \mathrm{m}(S), \mathrm{m}(S)+\mathrm{m}(\mathrm L(S)),\dots, \sum\nolimits_{i=1}^n \mathrm m(\mathrm L^i(S)),\to\right\},
\end{equation}
where $\mathrm L^i(S)$ is defined recursively as follows: $\mathrm L^0(S)=S$ and for every positive integer $i$, $\mathrm L^i(S)=\mathrm L(\mathrm L^{i-1}(S))$. The integer $n$ can be taken to be the minimum such that $\mathrm L^{n+1}(S)=\mathbb N$; and so $\mathrm m(\mathrm L^n(S))\ge 2$.

\section{Arf sequences}

We say that a sequence of integers $(x_1,\ldots,x_n)$ is an \emph{Arf sequence} provided that 
\begin{itemize}
\item $x_n\ge \cdots \ge x_1\ge 2$ and

\item  $x_{i+1}\in \{x_i,x_i+x_{i-1},\ldots, x_i+\cdots+x_1, \to\}$.
\end{itemize}

The following result (rephrased to our needs) supports this notation.

\begin{proposition}[{\cite[Corollary 39]{belga}}]\label{seq-ns}
Let $S$ be a nonempty proper subset of $\mathbb N$. Then $S$ is an Arf numerical semigroup if and only if there exists an Arf sequence $(x_1,\ldots, x_n)$ such that $S=\{0,x_n, x_n+x_{n-1},\ldots, x_n+\cdots+x_1,\to\}$.
\end{proposition}
\begin{proof}
Notice that in \cite[Corollary 39]{belga} the condition on $x_1$ is $x_1 \ge 1$. Notice that we can omit all $x_i=1$ in the sequence since in this way the resulting semigroup is the same, and we are considering the multiplicity sequence up to $\mathrm L^n(S)\neq \mathbb N$ and $\mathrm L^{n+1}(S)=\mathbb N$.	
\end{proof}
Given an Arf sequence $(x_1,\ldots,x_n)$, we will denote by $\mathrm S(x_1,\ldots, x_n)$ the associated Arf numerical semigroup given in Proposition \ref{seq-ns}, and we will say that it is the \emph{Arf numerical semigroup associated to $(x_1,\ldots, x_n)$}.

Hence for every Arf sequence $(x_1,\ldots,x_n)$, $\mathrm S(x_1,\ldots,x_n)$ is a numerical semigroup not equal to $\mathbb N$ and with the Arf property. And given an Arf numerical semigroup $S\neq \mathbb N$, according to \eqref{arf-mult-seq} and Proposition \ref{seq-ns}, if $n$ is a positive integer such that  $\mathrm L^n(S)\subsetneq \mathrm L^{n+1}(S)=\mathbb N$, the sequence $(\mathrm m(\mathrm L^n(S)), \mathrm m(\mathrm L^{n-1}(S)), \ldots, \mathrm m(S))$ is an Arf sequence.  This proves the following.

\begin{corollary}\label{bijection}
Let $\mathcal S$ be the set of Arf sequences, and let $\mathcal A$ be the set of all Arf numerical semigroups. The mapping  
\[
\begin{matrix}
\mathrm S: \mathcal S\to \mathcal A\setminus \{\mathbb N\},\\
	 (x_1,\ldots,x_n)\mapsto \mathrm S(x_1,\ldots, x_n)
\end{matrix}
\] is a bijection, and its inverse is the map $S\mapsto (\mathrm m(\mathrm L^n(S)), \mathrm m(\mathrm L^{n-1}(S)), \ldots, \mathrm m(S))$.
\end{corollary}

It is then clear that counting Arf numerical semigroups is tightly related to counting Arf sequences. Moreover, if we are looking for numerical semigroups with a prescribed genus or Frobenius number, the following result will be of great help.

\begin{proposition}\label{genus-frob-seq}
Let $(x_1,\ldots, x_n)$ be an Arf sequence. Then
\begin{enumerate}[(i)]
\item $\mathrm F(\mathrm S(x_1,\ldots,x_n))=x_1+\cdots + x_n-1$ (and thus $\mathrm c(\mathrm S(x_1,\ldots,x_n))=x_1+\cdots + x_n)$),
\item $\mathrm g(\mathrm S(x_1,\ldots,x_n))=x_1+\cdots+x_n-n$.
\end{enumerate}
\end{proposition}
\begin{proof}
In order to ease the notation, set $S=\mathrm S(x_1,\ldots, x_x)$, which by Proposition \ref{seq-ns} we know it is an Arf numerical semigroup. From the very construction of $S$, we have that the conductor of $S$ is at most $x_1+\cdots+x_n$.

\begin{enumerate}[(i)]
\item  From the above paragraph, it suffices to show that $x_1+\cdots+x_n-1\not \in S$. But this follows easily from the fact that $x_1\ge 2$.

\item We can explicitly write the set of gaps of $S$, 
\begin{multline*}
\mathbb N\setminus S=\{1,\ldots,x_n-1,x_n+1,\ldots, x_n+x_{n-1}-1,\ldots,x_n+\cdots+x_2-1,\\ x_n+\cdots+x_2+1,\ldots, x_n+\cdots+x_1-1\}.
\end{multline*}
It follows that $\mathrm g(S)= (x_n-1)+(x_{n-1}-1)+\cdots+(x_1-1)=x_1+\cdots+x_n-n$.\qedhere
\end{enumerate}
\end{proof}

\section{The set of Arf numerical semigroups with given conductor}

Let $c$ be a positive integer. In light of Corollary \ref{bijection} and Proposition \ref{genus-frob-seq}, in order to calculate the set of Arf numerical semigroups with conductor $c$ we only have to calculate some specific integer partitions of $c$, and then their images via the map $\mathrm S$. We can compute the set of integer partitions of $c$ with the help of \cite{partitions} or the built-in \texttt{GAP} command \texttt{partitions}, and either filter those having 1's or while constructing them avoid 1's in the partition. However, the number of partitions grows exponentially (for instance \texttt{NrPartitions(100)} in \texttt{GAP} yields 190569292), and then we must choose which partitions are Arf sequences. We do not have, as in the case of saturated numerical semigroups a ``next" function that, given an Arf sequence, computes the next in a prescribed ordering \cite{sat}.

In this section we present an alternative to the approach of computing all partitions and filter those that are Arf sequence. The procedure dynamically calculates the set of all Arf numerical semigroups with conductor less than or equal to $C$. The main idea is based on the following result, which allows to construct all Arf sequences of length $k+1$ from the set of Arf sequences of length $k$. Its proof follows directly from the definition of Arf sequence.

\begin{proposition}\label{seq-recur}
Let $\mathcal S$ be the set of all Arf sequences and let $k$ be a positive integer.
\begin{enumerate}[(i)]
\item If  $(x_1,\ldots, x_{k+1})\in \mathcal S$, then $(x_1,\ldots, x_k)\in \mathcal S$.
 
\item If $(x_1,\ldots, x_k)\in \mathcal S$, then $(x_1,\ldots, x_k,x_{k+1})\in \mathcal S$ for all $x_{k+1}\in \mathrm S(x_1,\ldots,x_k)^*$.
\end{enumerate}
\end{proposition}

Let us denote by $\mathcal S_k$ the set of all Arf sequences of length $k$, and for a positive integer $n$, set 
\[\mathcal S_k(n)=\{(x_1,\ldots,x_k)\in \mathcal S_k \mid x_1+\cdots+x_k\le n\}.\] 
As a consequence of the last result and that $x_1\ge 2$ we obtain the following. We use $(x,y]$, with $x,y\in \mathbb N$, to denote the interval of real numbers $r$ such that $x<r\le y$.

\begin{corollary}\label{seq-k-rec}
Let $\{b_i\}_{i\in\mathbb N}\subseteq \mathbb N$ be such that $0\le b_{i+1}-b_i\le 2$ for all $i\in \mathbb N$.
\begin{enumerate}[(i)]
\item $\mathcal S_1(b_1)= \big\{ (x_1) \mid x_1\in \{2,\ldots, b_1\}\big\}$.
\item For $k\in \mathbb N^*$, \[\mathcal S_{k+1}(b_{k+1})=\left\{ (x_1,\ldots, x_{k+1})~\middle\vert ~ \begin{matrix} (x_1,\ldots,x_k)\in \mathcal S_k(b_k)\\ x_{k+1}\in \mathrm S(x_1,\ldots,x_k)\cap (0,b_{k+1}-\sum_{i=1}^k x_i]\end{matrix}\right\}.\]
\end{enumerate}
\end{corollary}

In light of Proposition \ref{genus-frob-seq} (i) and Corollary \ref{seq-k-rec}, for the calculation of set of Arf numerical semigroups with conductor less than or equal to $c$ it is enough to calculate $\mathcal S_k(c)$ for $k\in\{1,\ldots,\lfloor c/2\rfloor\}$ (notice that the elements in an Arf sequence are greater than or equal to $2$). This is described in Algorithm \ref{alg:arf-frob-up-to}.

\begin{algorithm}[h]\caption{ArfNumericalSemigroupsWithConductorUpTo\label{alg:arf-frob-up-to}}

\KwData{A positive integer $c$}

\KwResult{The set of all Arf numerical semigroups with conductor less than or equal to $c$}

\For{$k\in \{1,\ldots, \lfloor c/2\rfloor \}$}{
Compute $\mathcal S_k(c)$  \tcc*{use Corollary \ref{seq-k-rec}}
}
$L=\bigcup_{k=1}^{\lfloor c/2\rfloor} \mathcal S_k(c)$\;
\Return $\{\mathbb N\}\cup\{ \mathrm S(x_1,\ldots, x_n) \mid (x_1,\ldots, x_n)\in L\}$
\end{algorithm}

\begin{figure}[h]
\begin{tabular}{|l|l||l|l||l|l||l|l|}\hline
$F$ & na($F$) & $F$ & na($F$) & $F$ & na($F$) & $F$ & na($F$) \\ \hline  \hline
1&1&26&111&51&1643&76&5494\\ 
2&1&27&176&52&1196&77&9215\\ 
3&2&28&138&53&2043&78&5707\\ 
4&2&29&239&54&1289&79&10469\\ 
5&4&30&150&55&2339&80&6709\\ 
6&3&31&298&56&1563&81&10822\\ 
7&7&32&211&57&2513&82&7698\\ 
8&6&33&341&58&1854&83&12951\\ 
9&10&34&268&59&3134&84&7705\\ 
10&9&35&440&60&1852&85&14028\\ 
11&17&36&279&61&3542&86&9399\\ 
12&12&37&535&62&2414&87&15011\\ 
13&25&38&389&63&3823&88&10395\\ 
14&20&39&616&64&2726&89&17538\\ 
15&32&40&448&65&4499&90&10381\\ 
16&27&41&778&66&2809&91&19147\\ 
17&49&42&490&67&5184&92&12425\\ 
18&34&43&936&68&3501&93&20048\\ 
19&68&44&642&69&5542&94&13988\\ 
20&49&45&1001&70&3866&95&23263\\ 
21&80&46&759&71&6645&96&13876\\ 
22&66&47&1300&72&3936&97&25560\\ 
23&118&48&808&73&7413&98&16839\\ 
24&77&49&1496&74&4992&99&26734\\ 
25&145&50&1028&75&7829&100&17903\\  \hline
\end{tabular}
\caption{Number of Arf numerical semigroups with Frobenius number up to 100}
\label{fig-table-frob-100}
\end{figure}

Figure \ref{fig-table-frob-100} shows the number of Arf numerical semigroups of Frobenius number up to 100. We already have many functions in the package \texttt{numericalsgps} computing families of numerical semigroups with a given Frobenius number, and thus we decided in our implementation of Arf numerical semigroups with given conductor to use the Frobenius number instead. The calculation of the table took 36 seconds on a laptop. However, we still do not know how many numerical semigroups there are with Frobenius number 100; so the approach of considering them all and filtering those that are Arf was rejected from the very beginning. For instance, for $F=35$, there are 292081 numerical semigroups; among these,  8959 have maximal embedding dimension and only 440 have the Arf property. Figure \ref{fig-comp-frob-saturated} compares the number of numerical semigroups with given Frobenius number that are saturated (as calculated in \cite{sat}) with those that are Arf. 
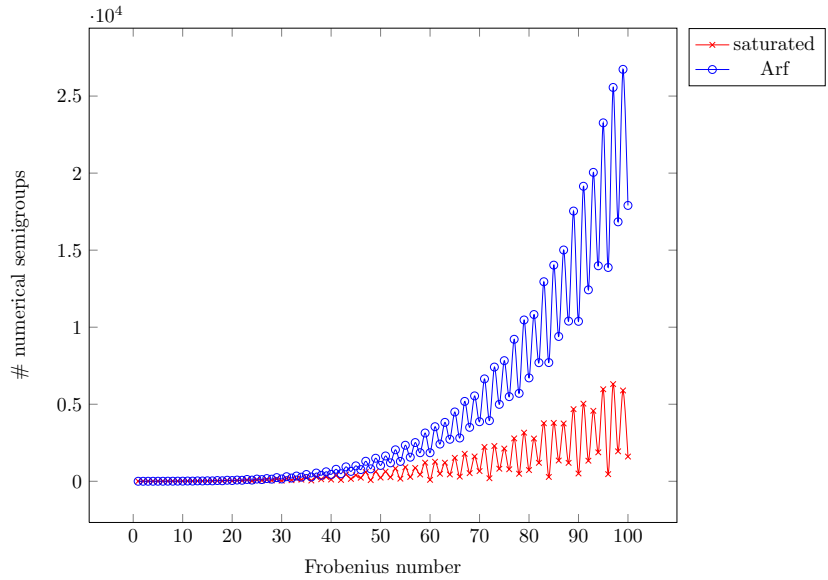
\begin{figure}

\begin{tikzpicture}[scale=.75]
\pgfplotsset{every axis legend/.append style={
at={(1.02,1)},
anchor=north west}}

    \begin{axis}[
	width=12cm,
        xlabel=Frobenius number,
        ylabel=\# numerical semigroups%,
	%legend pos=north west
]

    \addplot[smooth,color=red,mark=x]
        plot coordinates {
(1,1)
(2,1)
(3,2)
(4,2)
(5,4)
(6,3)
(7,7)
(8,5)
(9,9)
(10,8)
(11,16)
(12,7)
(13,21)
(14,14)
(15,25)
(16,18)
(17,39)
(18,16)
(19,50)
(20,22)
(21,52)
(22,40)
(23,84)
(24,20)
(25,92)
(26,53)
(27,103)
(28,54)
(29,144)
(30,39)
(31,175)
(32,68)
(33,166)
(34,105)
(35,240)
(36,49)
(37,280)
(38,131)
(39,285)
(40,113)
(41,378)
(42,88)
(43,439)
(44,155)
(45,389)
(46,233)
(47,597)
(48,79)
(49,624)
(50,239)
(51,628)
(52,266)
(53,828)
(54,170)
(55,909)
(56,284)
(57,865)
(58,440)
(59,1210)
(60,95)
(61,1267)
(62,490)
(63,1208)
(64,443)
(65,1522)
(66,303)
(67,1785)
(68,528)
(69,1612)
(70,662)
(71,2228)
(72,197)
(73,2291)
(74,816)
(75,2124)
(76,779)
(77,2783)
(78,491)
(79,3157)
(80,728)
(81,2775)
(82,1200)
(83,3765)
(84,282)
(85,3789)
(86,1347)
(87,3752)
(88,1196)
(89,4681)
(90,506)
(91,5039)
(92,1336)
(93,4574)
(94,1878)
(95,5973)
(96,463)
(97,6307)
(98,1944)
(99,5894)
(100,1605)
};
    \addlegendentry{saturated}

\addplot[smooth,mark=o,blue] plot coordinates {

(1,1)
(2,1)
(3,2)
(4,2)
(5,4)
(6,3)
(7,7)
(8,6)
(9,10)
(10,9)
(11,17)
(12,12)
(13,25)
(14,20)
(15,32)
(16,27)
(17,49)
(18,34)
(19,68)
(20,49)
(21,80)
(22,66)
(23,118)
(24,77)
(25,145)
(26,111)
(27,176)
(28,138)
(29,239)
(30,150)
(31,298)
(32,211)
(33,341)
(34,268)
(35,440)
(36,279)
(37,535)
(38,389)
(39,616)
(40,448)
(41,778)
(42,490)
(43,936)
(44,642)
(45,1001)
(46,759)
(47,1300)
(48,808)
(49,1496)
(50,1028)
(51,1643)
(52,1196)
(53,2043)
(54,1289)
(55,2339)
(56,1563)
(57,2513)
(58,1854)
(59,3134)
(60,1852)
(61,3542)
(62,2414)
(63,3823)
(64,2726)
(65,4499)
(66,2809)
(67,5184)
(68,3501)
(69,5542)
(70,3866)
(71,6645)
(72,3936)
(73,7413)
(74,4992)
(75,7829)
(76,5494)
(77,9215)
(78,5707)
(79,10469)
(80,6709)
(81,10822)
(82,7698)
(83,12951)
(84,7705)
(85,14028)
(86,9399)
(87,15011)
(88,10395)
(89,17538)
(90,10381)
(91,19147)
(92,12425)
(93,20048)
(94,13988)
(95,23263)
(96,13876)
(97,25560)
(98,16839)
(99,26734)
(100,17903)
};
  \addlegendentry{Arf}

    \end{axis}
    \end{tikzpicture}

\caption{Comparison with the number of saturated numerical semigroups}
\label{fig-comp-frob-saturated}
\end{figure}

\begin{example}
Let us compute the set of numerical semigroups with conductor less than or equal to six and with the Arf property. By Figure \ref{fig-table-frob-100} we already know that we have ten of them (eleven counting $\mathbb N$: 1+1+1+2+2+4; we have to go up to Frobenius number 5). As we have pointed above we must calculate $\mathcal S_k(6)$ for $k\in\{1,2,3\}$. 
\begin{itemize}
\item $\mathcal S_1(6)=\{ (2), (3), (4), (5), (6)\}$,
\item $\mathcal S_2(6)=\{ (2,2), (3,3), (2,3), (2,4)\}$,
\item $\mathcal S_3(6)=\{(2,2,2)\}$.
\end{itemize}
Now we have to translate these sequences to numerical semigroups via the map $\mathrm S$. For instance $\mathrm S(2,3)=\{0,3,5,\to\}=\langle 3,5,7\rangle$. We then obtain 
\begin{itemize}
\item $\langle 2,3\rangle$, $\langle 3,4,5\rangle$,  $\langle 4,5,6,7\rangle$, $\langle 5,6,7,8,9\rangle$, $\langle 6,7,8,9,10,11\rangle$,
\item $\langle 2,5\rangle$, $\langle 3,7,8\rangle$, $\langle 3,5,7\rangle$, $\langle 4,6,7,9\rangle$,
\item $\langle 2,7\rangle$.
\end{itemize}
Finally, we have to add $\mathbb N$.  In a \texttt{GAP} session with the package \texttt{numericalsgps} we would proceed as follows:
\begin{verbatim}
gap> la5:=ArfNumericalSemigroupsWithFrobeniusNumberUpTo(5);;
gap> List(la5,MinimalGeneratingSystem);                     
[ [1], [ 2, 3 ], [ 3 .. 5 ], [ 4 .. 7 ], [ 5 .. 9 ], [ 6 .. 11 ], 
  [ 2, 5 ], [ 3, 5, 7 ], [ 4, 6, 7, 9 ], [ 3, 7, 8 ], [ 2, 7 ] ]
\end{verbatim}
As we mentioned in the introduction, adjoining the Frobenius number to an Arf numerical semigroup yields another Arf numerical semigroup. Figure \ref{fig-hasse-frob-5} represents the Hasse diagram of all numerical semigroups conductor less than or equal to 6 and with the Arf property.
\end{example}

\begin{figure}
\includegraphics[width=.5\textwidth]{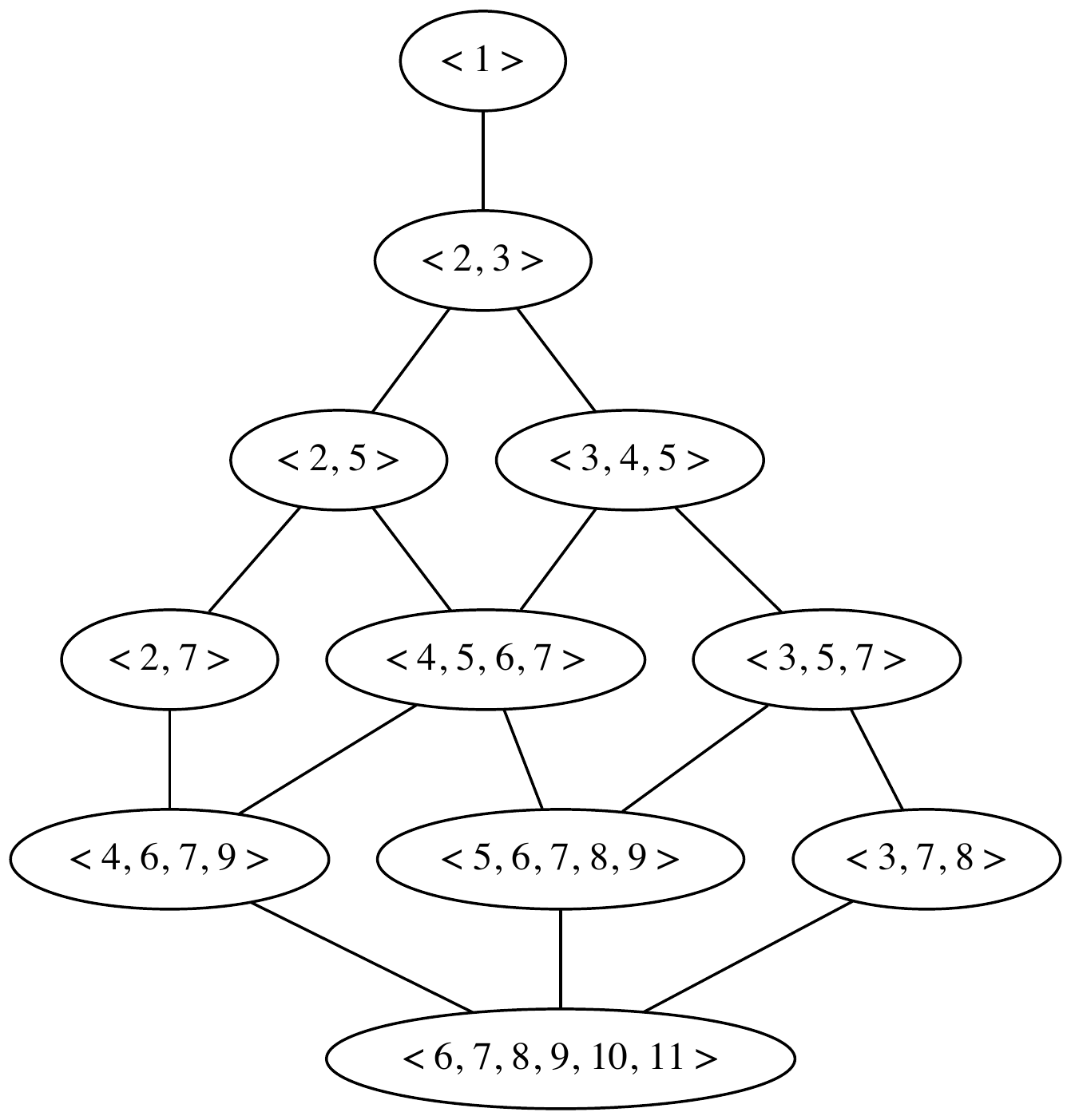}
\caption{Hasse diagram of Arf numerical semigroups with conductor up to six.}
\label{fig-hasse-frob-5}
\end{figure}

\section{Arf numerical semigroups with given genus}

As in the previous section we are again interested in Arf sequences with particular characteristics. In this case, by Proposition \ref{genus-frob-seq} (ii), the length of the sequence is also relevant.  If we fix the genus $g$, we need to calculate, for every suitable $k$, $\mathcal S_k(g+k)$, and then take the union of all of them. Also, in contrast to the conductor case, $k$ can range up to $g$, since $\mathrm g(\mathrm S(x_1,\ldots, x_n))= x_1+\cdots +x_n -n\le g$ and $x_i\ge 2$, forces $2n-n\le g$, that is, $n\le g$. 

In order to use recursion we must be able to construct $\mathcal S_{k+1}(g+k+1)$ from $\mathcal S_k (g+k)$. We can do this by using Corollary \ref{seq-k-rec} with $b_i=g+i$ for all $i\in\mathbb N$.

%\begin{corollary}\label{rec-genus}
%Let $n$ be a positive integer. 
%\begin{enumerate}[(i)]
%\item $\mathcal S_1(g+1)= \big\{ (x_1) \mid x_1\in \{2,\ldots, g+1\}\big\}$.
%\item For $k\in \mathbb N^*$, \[\mathcal S_{k+1}(g+k+1)=\left\{ (x_1,\ldots, x_{k+1})~\middle\vert ~ \begin{matrix} (x_1,\ldots,x_k)\in \mathcal S_k(g+k)\\ x_{k+1}\in \mathrm S(x_1,\ldots,x_k)\cap \left(0,g+k+1-\sum_{i=1}^k x_i\right]\end{matrix}\right\}.\]
%\end{enumerate}
%\end{corollary}

Algorithm \ref{alg:arf-genus-up-to} gathers the procedure to calculate all Arf numerical semigroups with genus up to $g$.

\begin{algorithm}\caption{ArfNumericalSemigroupsWithGenusUpTo\label{alg:arf-genus-up-to}}

\KwData{A positive integer $g$}

\KwResult{The set of all Arf numerical semigroups with genus less than or equal to $g$}

\For{$k\in \{1,\ldots, g \}$}{
Compute $\mathcal S_k(g+k)$  \tcc*{use Corollary \ref{seq-k-rec}}
}
$L=\bigcup_{k=1}^{g} \mathcal S_k(g+k)$\;
\Return $\{\mathbb N\}\cup\{ \mathrm S(x_1,\ldots, x_n) \mid (x_1,\ldots, x_n)\in L\}$
\end{algorithm}

\begin{example}\label{ex-lag5}
Let us apply the procedure in this section to calculate all Arf numerical semigroups with genus less than or equal to 5. We have to compute $\bigcup_{k=1}^g S_k(g+k)$. 
\begin{itemize}
\item $\mathcal S_1(5+1)= \{  ( 2 ), ( 3 ), ( 4 ), ( 5 ), ( 6 ) \}$, 
\item $\mathcal S_2(5+2)=  \{ ( 2, 2 ), ( 2, 3 ), ( 2, 4 ), ( 2, 5 ), ( 3, 3 ), ( 3, 4 ) \}$,
\item $\mathcal S_3(5+3)=\{  ( 2, 2, 2 ), ( 2, 2, 4 ), ( 2, 3, 4 ) \}$, 
\item $\mathcal S_4(5+4)= \{ ( 2, 2, 2, 2 ) \}$, 
\item $\mathcal S_5(5+5)=  \{ ( 2, 2, 2, 2, 2 ) \}$.
\end{itemize}

Next, we have to compute the image of each of them under $\mathrm S$, and finally add $\mathbb N$. In \texttt{GAP}, we can do this with the package \texttt{numericalsgps} as follows.
\begin{verbatim}
gap> lag5:=ArfNumericalSemigroupsWithGenusUpTo(5);;
gap> List(lag5,MinimalGeneratingSystem);
[ [ 1 ], [ 2, 3 ], [ 2, 5 ], [ 2, 7 ], [ 2, 9 ], [ 2, 11 ], 
  [ 4, 6, 9, 11 ], [ 3, 5, 7 ], [ 3, 8, 10 ], [ 4, 6, 7, 9 ], 
  [ 5, 7, 8, 9, 11 ], [ 3 .. 5 ], [ 3, 7, 8 ], [ 4, 7, 9, 10 ], 
  [ 4 .. 7 ], [ 5 .. 9 ], [ 6 .. 11 ] ]
\end{verbatim}
\end{example}

\begin{figure}
\begin{tabular}{|l|l||l|l||l|l||l|l|}\hline
$g$ & na($g$) & $g$ & na($g$) & $g$ & na($g$) & $g$ & na($g$) \\ \hline  \hline
1&1&26&251&51&2504&76&12275\\ 
2&2&27&284&52&2694&77&12979\\ 
3&3&28&317&53&2904&78&13701\\ 
4&4&29&355&54&3131&79&14468\\ 
5&6&30&393&55&3358&80&15295\\ 
6&8&31&433&56&3605&81&16114\\ 
7&10&32&487&57&3851&82&16959\\ 
8&13&33&538&58&4112&83&17840\\ 
9&17&34&594&59&4391&84&18765\\ 
10&21&35&658&60&4699&85&19738\\ 
11&26&36&721&61&5022&86&20781\\ 
12&31&37&793&62&5365&87&21864\\ 
13&36&38&866&63&5705&88&22993\\ 
14&47&39&946&64&6074&89&24163\\ 
15&55&40&1037&65&6472&90&25351\\ 
16&62&41&1138&66&6881&91&26581\\ 
17&74&42&1234&67&7307&92&27899\\ 
18&87&43&1338&68&7767&93&29246\\ 
19&101&44&1452&69&8240&94&30664\\ 
20&116&45&1584&70&8740&95&32139\\ 
21&133&46&1720&71&9265&96&33657\\ 
22&152&47&1861&72&9813&97&35228\\ 
23&174&48&2008&73&10386&98&36882\\ 
24&196&49&2164&74&10999&99&38602\\ 
25&222&50&2332&75&11620&100&40412\\ \hline
\end{tabular}
\caption{Number of Arf numerical semigroups with genus less than or equal to 100.}
\label{table-genus}
\end{figure}

Notice that in contrast to the sequence of the number of Arf numerical semigroups with given Frobenius number (Figure \ref{fig-table-frob-100}), we see in Figure \ref{table-genus}, that in the case of counting with respect to the genus, the resulting sequence is increasing.

From \cite[Section~2]{arf-num-sem}, we know that the tree of Arf numerical semigroups is a binary tree. This tree is constructed as follows.

Let $A$ be a nonempty set of positive integers with greatest common divisor one. The intersection of all Arf numerical semigroups containing $A$ is an Arf semigroup (every numerical semigroup containing $A$ must also contain $\langle A\rangle$; whence there are only finitely many containing $A$). We denote by $\mathrm{Arf}(A)$ this numerical semigroup. 

Given an Arf numerical semigroup $S$, we say that $A$ is an Arf system of generators of $S$ if $\mathrm{Arf}(A)=S$; and it is a minimal Arf system of generators if no proper subset of $A$ has this property. The elements of $A$ are called minimal Arf generators of $S$. 

The tree of Arf numerical semigroups is constructed recursively by removing  minimal Arf generators greater than the Frobenius number for each semigroup in the tree. Lemma 12 in \cite{arf-num-sem} states that at most two minimal Arf generators are greater than the Frobenius number of an Arf semigroup, and according to its proof these are the Frobenius number plus one and plus two. This is why the tree is binary. Also, a leaf in this tree is an Arf numerical semigroup with no minimal Arf generators above its Frobenius number.

The absence of leafs would explain the increasing of the sequence in Figure \ref{table-genus}. However this is not the case, there are plenty of leaves in this tree. For instance, $\mathrm F(\mathrm{Arf}(5,7))=8$, and consequently $\mathrm{Arf}(5,7)$ is a leaf in the binary tree of Arf numerical semigroups (this example has genus 6, all the semigroups appearing in Example \ref{ex-lag5} have descendants in the tree). Figure \ref{fig:bin-tree-6} depicts the binary tree of Arf numerical semigroups up to genus 6; the shaded node corresponds to the unique leaf in the tree of all Arf numerical semigroups. Each layer corresponds to a different genus.

%\textcolor{red}{Thus in order to proof that the sequence of the number Arf numerical semigroups with given genus is increasing, it should be shown that at each layer there are more numerical semigroups with two descendants than leaves. In other words, fixed $g$ a positive integer in the set of Arf numerical semigroups with genus $g$ there are more semigroups $S$ such that $\mathrm F(S)+1$ and $\mathrm F(S)+2$ are Arf minimal generators than semigroups $S$ such that neither $\mathrm F(S)+1$ nor $\mathrm F(S)+2$ belongs to the Arf minimal system of generators of $S$. According to the above notations, for an Arf sequence, we are wondering when $x_1+\cdots+x_n, x_1+\cdots+x_n+1\in \mathrm{Arf}(x_n,x_n+x_{n-1}, \ldots, x_n+\cdots+x_2)$}.
\begin{figure}
\centering
\includegraphics[width=\textwidth]{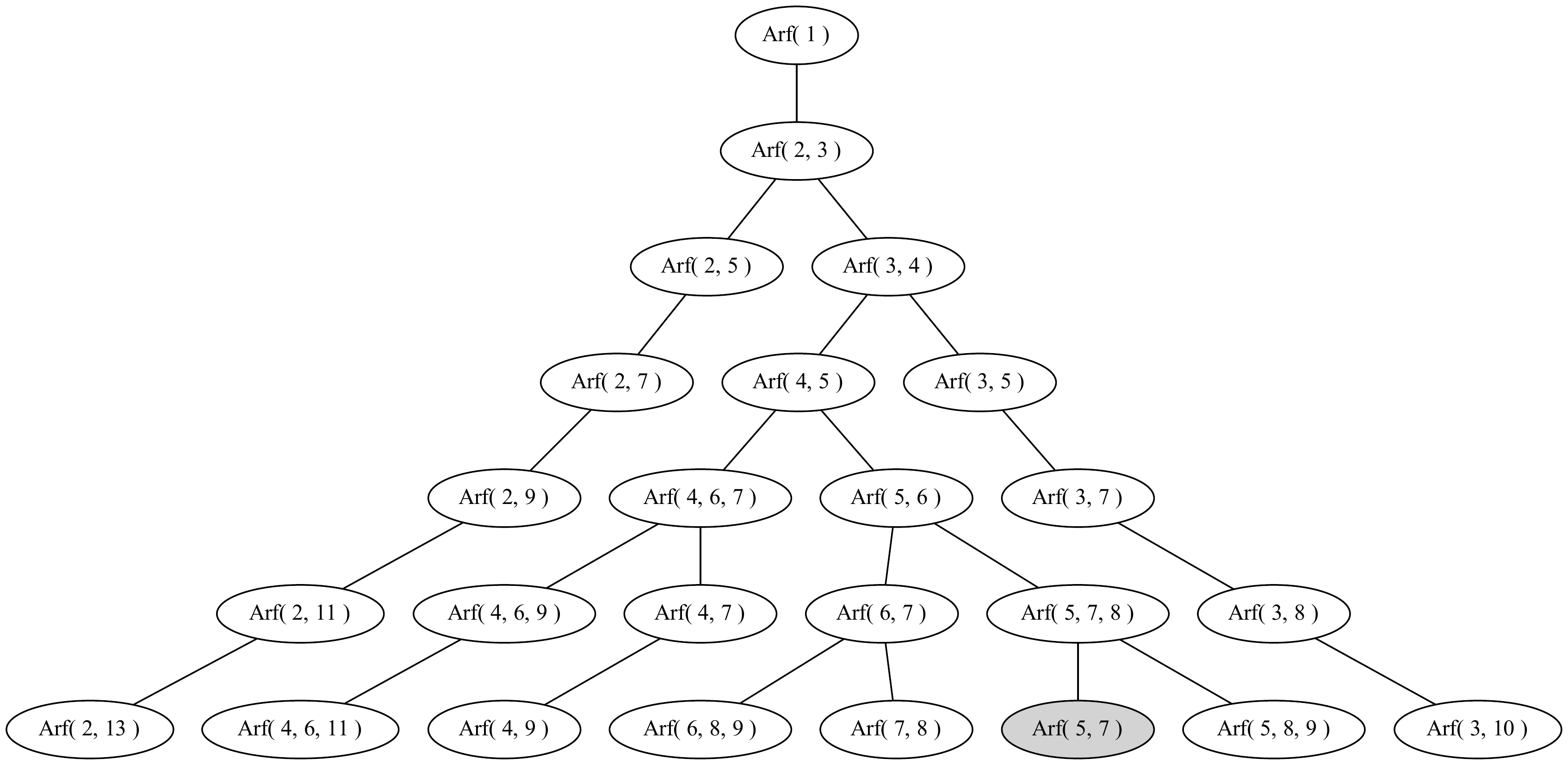}
\caption{The binary tree of Arf numerical semigroups of genus up to 6.}
\label{fig:bin-tree-6}
\end{figure}

\section{Fixing the genus and the conductor}

In this section we are interested in calculating the set of all Arf numerical semigroups with fixed genus $g$ and conductor $c$. It is well known, and easy to prove, that if $S$ is a numerical semigroup, then $2\mathrm g(S)\ge \mathrm c(S)$ (see for instance \cite[Lemma 2.14]{ns}). From the definition of genus and Frobenius number it also follows easily that $\mathrm g(S)\le \mathrm F(S)$. The only numerical semigroup with genus equal to zero is $\mathbb N$; so we may assume that 
\[
1\le g\le c-1<2g.
\] 
As a consequence of Proposition \ref{genus-frob-seq}, we have the following restriction on the lengths of Arf sequences yielding semigroups with prescribed genus $g$ and conductor $c$.

\begin{corollary}
Let $(x_1,\ldots,x_n)$ be an Arf sequence and let $S=\mathrm S(x_1,\ldots,x_n)$. Then $n=\mathrm c(S)-\mathrm g(S)$.
\end{corollary}

\begin{algorithm}[h]
\KwData{Positive integers $g$ and $c$ with $1\le g\le c-1<2g$}

\KwResult{The set of all Arf numerical semigroups with genus $g$ and conductor $c$}

$n=c-g$\;
$L_1=\{ (x_1)\mid x_1\in \{2,\ldots, \lfloor c/n\rfloor\}$\;
\For{$k\in \{2,\ldots,  n \}$}{
$L_{k}=\left\{(x_1,\ldots,x_k)~\middle \vert~ \begin{matrix}
(x_1,\ldots,x_{k-1})\in L_{k-1},  x_k\in\mathrm S(x_1,\ldots,x_{k-1})^* \\ x_1+\cdots+x_{k-1}+(n-(k-1))x_k\le c\end{matrix}\right\}$\;
}
$L=\{(x_1,\ldots,x_n)\in L_n\mid x_1+\cdots+x_n=c\}$\;
\Return $\{ \mathrm S(x_1,\ldots, x_n) \mid (x_1,\ldots, x_n)\in L\}$
\caption{ArfNumericalSemigroupsWithGenusAndFrobeniusNumber\label{alg:arf-genus-frob}}
\end{algorithm}

The correctness of Algorithm \ref{alg:arf-genus-frob} follows from the next two observations. If $(x_1,\ldots,x_n)$ is an Arf sequence and $\mathrm c(\mathrm S(x_1,\ldots, x_n))=c$, then by Proposition \ref{genus-frob-seq}, $c=x_1+\cdots +x_n$.
\begin{enumerate}[(i)]
\item As $x_n\ge \cdots \ge x_1$, we deduce $nx_1\le c$. This implies $x_1\le \lfloor c/n\rfloor$.

\item Also from $x_n\ge \cdots\ge x_k$, we deduce $x_1+\cdots +x_{k-1}+(n-(k-1))x_k\le c$. 
\end{enumerate}

Figure \ref{fig:plot3d} depicts Arf numerical semigroups with genus $g$ ranging from 1 to 20 and conductor from $g+1$ to $2g$.

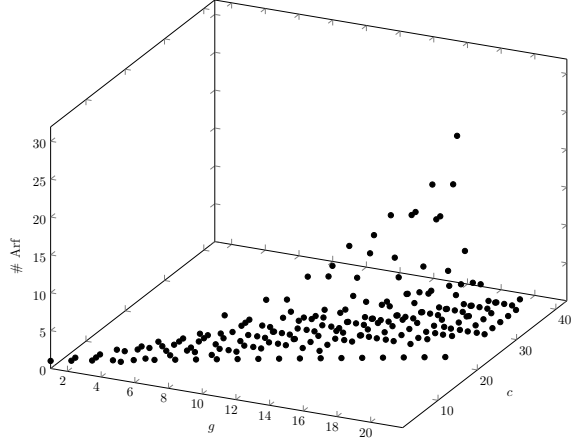
\begin{figure}[h]
\centering

\begin{tabular*}{.4\linewidth}{@{\extracolsep{\fill}}|r|rrrrrrrrrrrrrrrrrrrr}
$g$\\ \cline{1-1}
1& 1\\ 
2& 1& 1\\ 
3& 1& 1& 1\\ 
4& 1& 2& 0& 1\\ 
5& 1& 2& 2& 0& 1\\ 
6& 1& 3& 2& 1& 0& 1\\ 
7& 1& 3& 3& 1& 1& 0& 1\\ 
8& 1& 4& 3& 3& 0& 1& 0& 1\\ 
9& 1& 4& 6& 1& 3& 0& 1& 0& 1\\ 
10& 1& 5& 5& 5& 1& 2& 0& 1& 0& 1\\ 
11& 1& 5& 8& 4& 3& 1& 2& 0& 1& 0& 1\\ 
12& 1& 6& 8& 6& 2& 4& 0& 2& 0& 1& 0& 1\\ 
13& 1& 6& 11& 5& 5& 0& 4& 0& 2& 0& 1& 0& 1\\ 
14& 1& 7& 11& 12& 3& 5& 1& 3& 0& 2& 0& 1& 0& 1\\ 
15& 1& 7& 15& 8& 10& 2& 4& 1& 3& 0& 2& 0& 1& 0& 1\\ 
16& 1& 8& 14& 16& 4& 6& 1& 5& 0& 3& 0& 2& 0& 1& 0& 1\\ 
17& 1& 8& 19& 13& 10& 4& 7& 0& 5& 0& 3& 0& 2& 0& 1& 0& 1\\ 
18& 1& 9& 19& 19& 8& 11& 1& 7& 1& 4& 0& 3& 0& 2& 0& 1& 0& 1\\ 
19& 1& 9& 23& 18& 18& 3& 10& 1& 6& 1& 4& 0& 3& 0& 2& 0& 1& 0& 1\\ 
20& 1& 10& 23& 29& 9& 13& 4& 8& 1& 7& 0& 4& 0& 3& 0& 2& 0& 1& 0& 1\\  
  \multicolumn{21}{c}{conductor from $g+1$ to $2g$} 
\end{tabular*} \phantom{lalalalaaaaaaaaaa}
\begin{tikzpicture}[scale=.5]
\begin{axis}[xlabel={$g$},ylabel={$c$},zlabel={\# Arf},
	xmin=1,xmax=20,
	ymin=1,ymax=39,
	zmin=0,zmax=29,
	enlargelimits=upper,
	ytick={0,10,...,40}
]

	% `pgfplots_scatter4.dat' contains a large sequence of
	% the form
	% l_0   l_1     l_2     
	% 1     6       -1      
	% -1    -1      -1      
	\addplot3[only marks, mark size=2] 
		file {data-g-f.dat};
\end{axis}
\end{tikzpicture}

\caption{Number of Arf numerical semigroups with genus $g$ and conductor $c$}
\label{fig:plot3d}
\end{figure}

\section{Kunz coordinates of Arf numerical semigroups}

Let $S$ be a numerical semigroup and $s\in S^*$. Recall that the \emph{Ap\'ery set} of $s$ in $S$ is defined as 
\[
\mathrm{Ap}(S,s)=\{ n\in S\mid n-s\not\in S\}.
\] 
It is well known (see for instance \cite[Chapter 1]{ns}) that 
\begin{equation}\label{eq:ap}
\begin{matrix}
\mathrm{Ap}(S,s)=\{w(0)=0,w(1),\ldots, w(s-1)\},\\
w(i)=\min\{n \in S\mid n\bmod s=i\},\ i\in\{1,\ldots, s-1\}.
\end{matrix} 
\end{equation}

Observe that for every $z\in \mathbb Z$, there exists $k\in \mathbb N$ and $i\in\{0,\ldots,s-1\}$ such that $z=ks+w(i)$. Moreover, $z\in S$ if and only if $k\ge 0$ (see \cite[Chapter 1]{ns}). We will use this well known fact implicitly in this section.

If in addition $S$ is an Arf numerical semigroup and $m$ is its multiplicity, then we know that $S$ has maximal embedding dimension and thus $(\mathrm{Ap}(S,m)\setminus\{0\})\cup\{m\}=\{m,w(1),\ldots, w(m-1)\}$ is the minimal generating system of $S$.

For every $k\in \{1,\ldots, m-1\}$, $w(k)=x_km+k$ for some positive integer $x_k$. We say that $(x_1,\ldots, x_{m-1})$ are the \emph{Kunz coordinates} of $S$ \cite{kunz}. Notice that in this definition we can take $k=0$ and obtain $x_0=0$. We are not including $x_0$ in the sequence of Kunz coordinates because it is always $0$. We will use this implicitly in what follows.

We will fix the multiplicity, $m$, and for an integer $i$, we will write \[\overline{i} =i\bmod m\] (the remainder of the division of $i$ by $m$). 

Every Arf numerical semigroup has maximal embedding dimension, and thus its Kunz coordinates fulfill the following system of inequalities \cite{ns-coord}.
\begin{equation}\label{eq:med}
\begin{array}{cc}
x_i\geq 1 & \hbox{for all } i\in \{1,\ldots,m-1\},\\
x_i+x_j-x_{i+j}\geq 1 & \hbox{for all } 1\leq i\leq j\leq m-1, i+j\leq m-1,\\
x_i+x_j-x_{\overline{i+j}}\geq 0 &\hbox{for all } 1\leq i\leq j\leq m-1, i+j> m.
\end{array}
\end{equation}

Notice also that for every $i,j\in \{0,\ldots, m-1\}$, $w(i)+w(j)\in S$ and $w(i)+w(j)\equiv i+j \pmod m$. Hence $w(j)+w(j)=km+w(\overline{i+j})$. This implies that $(w(i)+w(j)-w(\overline{i+j}))/m\in \mathbb N$. So we define the \emph{cocycle} of $S$ with respect to $m$ as 
\[
\mathrm h(i,j)=\frac{w(i)+w(j)-w(\overline{i+j})}{m}.
\]

Next we see how the Arf condition is written in terms of cocycles.

Given a rational number $q$, denote by 
\[\lceil q\rceil = \min (\mathbb Z\cap [q,\infty)),\ \lfloor q\rfloor =\max (\mathbb Z\cap (-\infty,q]).\]

\begin{lemma}\label{ArfCocycleToCocycle}
Let \(S\) be a numerical semigroup with multiplicity $m$. Then $S$ is an Arf numerical semigroup if and only if for every $i,j\in\{0,\ldots, m-1\}$, 
\begin{enumerate}[(i)]
\item if $\lceil (w(i)-w(j))/m\rceil \ge 0$, 
\begin{equation}\label{eq:ArfCondToCocycle-g}
\mathrm h(j,j)-\mathrm h(\overline{2j-i}  , i) + 2\left\lceil \frac{w(i)-w(j)}{m}\right\rceil \geq 0;
\end{equation}
\item if $\lceil (w(i)-w(j))/m\rceil < 0$, 
\begin{equation}\label{eq:ArfCondToCocycle-l}
\mathrm h(j,j)-\mathrm h(\overline{2j-i}  , i) + \left\lceil \frac{w(i)-w(j)}{m}\right\rceil \geq 0.
\end{equation}
\end{enumerate}
\end{lemma}
\begin{proof}
Suppose \(S\) is an Arf numerical semigroup. 
For any \(i,j\in\{0, \ldots, m-1\}\), define 
\(t=\lceil \frac{w(i)-w(j)}{m}\rceil\). Then $w(j)+mt \geq w(i)$.

If $w(i)\ge w(j)$, then $w(j)+mt, w(i)\in S$.  By the Arf property \(2w(j)+2mt - w(i)\in S\). This element can be expressed as 
\[
2w(j) - w(i) + 2tm=w(\overline{2j-i}  ) + (\mathrm h(j,j) - \mathrm h(\overline{2j-i}  ,i) + 2t)m,
\]
and so it belongs to \(S\) if and only if \(\mathrm h(j,j)-\mathrm h(\overline{2j-i}  ,i) + 2t \geq 0\).

Now if $w(i)<w(j)$, then $w(j)\ge w(i)-mt$, and $w(j),w(i)-mt\in S$. By the Arf property we deduce in this case that $2w(j)+tm-w(i)\in S$. Arguing as above we deduce that this occurs if and only if $h(j,j) - h(\overline{2j-i}  ,i) + t \geq 0$.

For the converse, let \(x\geq y \in S\). We can write \(x=w(j)+am, y=w(i)+bm\) for some $i,j\in \{0,\ldots, m-1\}$ and \(a,b\in \mathbb N\). Put again \(t=\lceil\frac{w(i)-w(j)}{m}\rceil\). Then
\[
2x-y = 2w(j)+2am - w(i)-bm=
w(\overline{2j-i}  ) + (\mathrm h(j,j)-\mathrm h(\overline{2j-i}  , i) + 2a-b)m.
\]
By the condition \(x\geq y\), we have $w(j)+am\ge w(i)+bm$, and consequently $a-b\ge t$. 

If $t\ge 0$, as $2a-b\geq 2(a-b) \geq 2t$, we deduce $\mathrm h(j,j)-\mathrm h(\overline{2j-i}  , i) + 2a-b\ge \mathrm h(j,j)-\mathrm h(\overline{2j-i}  , i) + 2t$, which by \eqref{eq:ArfCondToCocycle-g} is nonnegative. This forces $2x-y\in S$.

If $t<0$, then $2a-b=a+a-b\ge a-b\ge t$. Arguing as in the preceding case, \eqref{eq:ArfCondToCocycle-g} ensures that $2x-y\in S$.
\end{proof}

Let us now translate cocycles to the language of Kunz coordinates.

\begin{lemma}\label{CocycleFromKunz}
Let $S$ be a numerical semigroup with multiplicity $m$ and Kunz coordinates $(x_1,\ldots,x_{m-1})$. 
\[
\mathrm h(i,j) = x_i+x_j-x_{\overline{i+j}}+\left\lfloor \frac{i+j}m\right\rfloor.
%\left\{
%\begin{array}{ll}
%x_i + x_j - x_{i+j} & \mbox{ if } i + j <m \\
%x_i + x_j - x_{i+j-m} + 1 & \mbox{ if } i + j \geq m
%\end{array}\right.
\]
\end{lemma}
\begin{proof}
By definition $\mathrm h(i,j)=(w(i)+w(j)-w(\overline{i+j}))/m$. This can be expressed in terms of Kunz coordinates as $\mathrm h(i,j)=(x_im+i+x_jm+j-x_{\overline{i+j}}m-\overline{i+j})/m= x_i+x_j-x_{\overline{i+j}}+(i+j-\overline{i+j})/m$. The proof now follows from the equality $(i+j-\overline{i+j})/m=\lfloor (i+j)/m\rfloor$.
\end{proof}

Notice also that 
\begin{equation}
\label{eq:diff-ap}
\lceil (w(i)-w(j))/m\rceil=x_i-x_j+\lceil (i-j)/m\rceil
\end{equation}
With this and Lemma~\ref{CocycleFromKunz} we can translate Lemma \ref{ArfCocycleToCocycle} to Kunz coordinates.

\begin{proposition}
Let $S$ be sa numerical semigroup with multiplicity $m$ and Kunz coordinates $(x_1,\ldots, x_{m-1})$. Then $S$ has the Arf property if and only if for any \(i,j\in \{0,\ldots, m-1\}\), 

\begin{enumerate}[(i)]
\item if $x_i+\lceil (i-j)/m\rceil\ge x_j$, 
\begin{equation}\label{eq:from-cocycle-cond-g}
  x_i - x_{\overline{2j-i}} + 2\left\lceil\frac{i -j}{m}\right\rceil + \left\lfloor \frac{2j}m\right\rfloor -\left\lfloor \frac{\overline{2j-i}+i}m\right\rfloor \geq 0;
\end{equation}
\item if $x_i+\lceil (i-j)/m\rceil< x_j$,
\begin{equation}\label{eq:from-cocycle-cond-l}
  x_j - x_{\overline{2j-i}} + \left\lceil\frac{i -j}{m}\right\rceil + \left\lfloor \frac{2j}m\right\rfloor -\left\lfloor \frac{\overline{2j-i}+i}m\right\rfloor \geq 0.
\end{equation}
\end{enumerate} 
\end{proposition}
\begin{proof}
From Lemma \ref{CocycleFromKunz} and  \eqref{eq:diff-ap} we obtain
\begin{align*}
\mathrm h(j,j)-\mathrm h(\overline{2j-i},i)+2\left\lceil \frac{w(i)-w(j)}m\right\rceil =\ & 2x_j-x_{\overline{2j}}+\lfloor 2j/m\rfloor\\
&  - x_{\overline{2j-i}}-x_i+x_{\overline{2j}}-\lfloor (\overline{2j-i}+i)/m\rfloor \\ 
& + 2(x_i-x_j)+2\lceil (i-j)/m\rceil\\
=\ & x_i-x_{\overline{2j-i}}+ 2\lceil (i-j)/m\rceil \\ & + \lfloor 2j/m\rfloor -\lfloor (\overline{2j-i}+i)/m\rfloor,
\end{align*}
and
\begin{align*}
\mathrm h(j,j)-\mathrm h(\overline{2j-i},i)+\left\lceil \frac{w(i)-w(j)}m\right\rceil =\ & 2x_j-x_{\overline{2j}}+\lfloor 2j/m\rfloor\\
&  - x_{\overline{2j-i}}-x_i+x_{\overline{2j}}-\lfloor (\overline{2j-i}+i)/m\rfloor \\ 
& + (x_i-x_j)+\lceil (i-j)/m\rceil\\
=\ & x_j-x_{\overline{2j-i}}+ \lceil (i-j)/m\rceil \\ & + \lfloor 2j/m\rfloor -\lfloor (\overline{2j-i}+i)/m\rfloor.
\end{align*}
Now we apply Lemma \ref{ArfCocycleToCocycle} and we are done.
\end{proof}

\section{Arf numerical semigroups with low multiplicity}

In this section we focus on the parametrization of Arf numerical semigroups with multiplicity up to six and given conductor.  To this end, we need some preliminary results.

The following well known result will be used to provide upper bounds for the conductor of a numerical semigroup.

\begin{lemma}[\cite{arf-num-sem}] {\label{2:2}}
Let $S$ be an Arf numerical semigroup with conductor $c$, and let $s$ be any element of $S$. If $s+1 \in S$, then  $s+k \in S$ for all $k \in {\mathbb N}$ and thus $c \leq s$.
\end{lemma}

By Selmer's formulas (see for instance \cite[Chapter 1]{ns}), we know that the Frobenius number of $S$ is $\max\mathrm{Ap}(S,m)-m$. 

\begin{lemma} {\label{2:3}}
Let $S$ be an Arf numerical semigroup with multiplicity $m$ and conductor $c$. For each $j \in\{2,3, \ldots , m-1\}$, we have
\begin{enumerate}[(i)]
\item if $w(j-1) < w(j)$, then $c \leq w(j)-1$,
\item if $w(j) < w(j-1)$, then $c \leq w(j-1)$.
\end{enumerate}
\end{lemma}
\begin{proof}
\emph{(i)} If $w(j-1) < w(j)$, then  $w(j)-w(j-1)-1$ is a nonnegative multiple of $m$ and therefore it is an element of $S$. Thus $w(j-1)+(w(j)-w(j-1)-1)= w(j)-1$ and $w(j)$ are both elements of $S$. Lemma {\ref{2:2}} forces $c \leq w(j)-1$.

\emph{(ii)} If $w(j) < w(j-1)$, then, as above, $w(j-1)-w(j)+1 \in S$. Thus $w(j) + (w(j-1)-w(j)+1) = w(j-1)+1$ and $w(j)$ are both elements of $S$.  So, $c \leq w(j-1)$ by Lemma {\ref{2:2}}.
\end{proof}

%We say that a numerical semigroup $S$ is \emph{ordinary} if there exists $m\in \mathbb N$ such that $S=\{0\}\cup (m+\mathbb N)$. These semigroups have always the Arf property. 

This has the corresponding consequence on Kunz coordinates.

\begin{corollary}\label{cor:change}
 Let $S$  be an Arf  numerical semigroup with multiplicity $m$, conductor $c$ and Kunz coordinates $(x_1, \ldots , x_{m-1})$. For every $i\in\{2,\ldots, m-1\}$, 
\begin{enumerate}[(i)]
\item if $x_{i-1} \leq x_i$, then $\frac{c-i+1}{m} \leq x_i$;
\item if $x_i \leq x_{i-1}$, then $\frac{c-i+1}{m} \leq x_{i-1}$.
\end{enumerate}
\end{corollary}
\begin{proof}
If $x_{i-1} \leq x_i$, then $x_{i-1}m+i-1 \leq x_im+i-1$, whence  $w(i-1) < w(i).$  Therefore, $c \leq w(i)-1$ by Lemma \ref{2:3}. Hence 
$c \leq x_im+i-1$ and $\frac{c-i+1}{m} \leq x_i$. The proof of \emph{(ii)} is similar.  
\end{proof}

Let $S$ be a numerical semigroup with multiplicity $m$ and conductor $c$. As every nonnegative multiple of $m$ is in $S$ and $c-1\not\in S$, it follows that $c\not\equiv 1\!\pmod m$. 

The following lemma shows that $w(1)$ and $w(m-1)$ are completely determined by the conductor and $m$ in an Arf numerical semigroup.

\begin{lemma} {\label{2:4}}
Let $S$ be an Arf numerical semigroup  with multiplicity $m$ and conductor $c$.
\begin{enumerate}[(i)]
\item $w(1)= \begin{cases}

 c+1 & {\mbox{if }} c \equiv 0\!\pmod m,\\
      c-\overline{c}+m+1 &  \hbox{otherwise}.
                 \end{cases}$

\item $w(m-1) = c-\overline{c}+m-1.$
\end{enumerate}
\end{lemma}
\begin{proof}
We know that $\overline{c} \in \{0, 2, \ldots , m-1\}$.

Let us first consider the case  $\overline{c}=0$. Since $ms \in S$ for all $s \in {\mathbb{N}}$, $ms + 1 \not \in S$ and $ms + m - 1 \not \in S$ for $s < \frac{c}{m}$ by  Lemma {\ref{2:2}}. Hence $w(1) = m \cdot \frac{c}{m} + 1 = c+1$ and $w(m-1) = m \cdot \frac{c}{m} + m - 1=c+m-1$. This proves \emph{(i)} and \emph{(ii)} when $\overline{c}=0$. 

It remains to prove \emph{(i)} and \emph{(ii)} for the case  $\overline{c} \neq 0$.  Again since  $ms \in S$ for all $s \in {\mathbb{N}}$, Lemma {\ref{2:2}} implies that $ms +  1 \not \in S$ for $s \leq \frac{c-\overline{c}}{m}$ and $ms + m - 1 \not \in S$ for $s < \frac{c-\overline{c}}{m}$. Therefore,
$w(1) = m \cdot (\frac{c-\overline{c}}{m}+1) + 1 = c-\overline{c}+m+1$ and $w(m-1) = m \cdot \frac{c-\overline{c}}{m} + m - 1= c-\overline{c}+m - 1$ when $\overline{c} \neq 0$. This completes the proof.
\end{proof}

Let us translate Lemma \ref{2:4} to the language of Kunz coordinates. \begin{corollary}\label{cor:eq-c}
Let $S$ be an Arf numerical semigroup with multiplicity $m$ and conductor $c$. Then,
\begin{equation}\label{eq:eq-c}
x_1=\left\lceil \frac{c}m\right\rceil,\ x_{m-1}=\left\lfloor \frac{c}m\right\rfloor.
\end{equation}
\end{corollary}
\begin{proof}
If $\overline{c}=0$, then we know that $w(1)=x_1m+1=c+1$, whence $x_1=c/m$. Also, $w(m-1)=x_{m-1}m+m-1=c+m-1$. Hence $x_{m-1}=c/m$.

If $\overline{c}\neq 0$, then $w(1)=x_1m+1=c-\overline{c}+m+1$. Thus, $x_1=(c-\overline{c})/m+1$. In this setting, $w(m-1)=x_{m-1}m+m-1=c-\overline{c}+m-1$. Hence $x_{m-1}=(c-\overline{c})/m$.
\end{proof}

\begin{lemma} {\label{2:5}}
Let $S$ be an Arf numerical semigroup  with multiplicity $m >2$.
For any integer $k$ with $0 < k < \frac{m}{2}$, we have
\[w(2k) \leq w(k)+k \ {\mbox{and}} \  w(m-2k) \leq w(m-k) + m-k.\]
\end{lemma}
\begin{proof}
Let $m>2$ and let $0 < k < \frac{m}{2}$. Note that $w(k)-k$ is a (non negative) multiple of $m$ and thus it is an element of $S$. Therefore  $2w(k)-(w(k)-k)= w(k) + k \in S$ by the Arf property. This implies $w(2k) \leq w(k)+k$ since $w(k)+k \equiv 2k\!\pmod m$. Similarly, $2w(m-k)-(w(m-k)-(m-k))=w(m-k)+(m-k) \in S$ which implies $w(m-2k) \leq w(m-k)+(m-k)$.
\end{proof}

%
%Let now $ m/2\ge k<m$. As above, $w(k)-k\in S$, and so $2w(k)-(w(k)-k)=w(k)+k\in S$. Notice that $w(k)+k\equiv 2k\pmod m$. Hence $w(2k-m)\le w(k)+k$. This means that $x_{2k-m}m+2k-m\le x_km+2k$. Therefore $x_{2k-m}\le x_k+1$. 2k'-m=2(m-k)-m=m-2k

As a consequence of Lemma \ref{2:5}, in the Arf setting, we can add more inequalities to the above system of inequalities.

\begin{corollary}\label{cor:more-eq}
Let $S$ be an Arf numerical semigroup with Kunz coordinates $(x_1,\ldots, x_n)$. Then for every integer $k$ with $0<k<\frac{m}2$,
\begin{equation}\label{eq:more-eq}
x_{2k}\le x_k \hbox{ and } x_{m-2k}\le x_{m-k}+1.
\end{equation}
\end{corollary}
\begin{proof}
Notice that $w(2k)=x_{2k}m+2k\le w(k)+k=(x_km+k)+k$, and so $x_{2k}\le x_k$.

For the other inequality, observe that $w(m-2k)=x_{m-2k}m+m-2k\le w(m-k)+m-k=(x_{m-k}m+m-k)+m-k= (x_{m-k}+1)m+m-2k$, and consequently $x_{m-2k}\le x_{m-k}+1$.
\end{proof}

%\textcolor{blue}{Now that we have these corollaries, which were not in the initial version, the following subsections could vary, or some parts could be rewritten accordingly. PLEASE CHECK!}

\subsection{Arf numerical semigroups with multiplicity one}
The only numerical semigroup with multiplicity one is $\mathbb N$, which is trivially Arf.
 
\subsection{Arf numerical semigroups with multiplicity two}
Numerical semigroups with multiplicity $2$ are completely determined by their conductor. In fact, if $S$ is a numerical semigroup with multiplicity $2$ and conductor $c$, then $c$ is an even number and $S = \langle  2, c+1 \rangle$. It is easily seen by directly applying the Arf pattern that every numerical semigroup with multiplicity $2$ is an Arf numerical semigroup.

\subsection{Arf numerical semigroups with multiplicity three}
Numerical semigroups with multiplicity $3$ or more are not completely determined by their conductor alone. The genus is needed to completely determine them ~\cite{tres-cuatro}. In that paper, formulas for the number of numerical semigroups with multiplicity $3$ having a prescribed Frobenius number or genus are given.

As we see next, if the Arf property is assumed, then the semigroup is fully determined by the multiplicity and the conductor.

\begin{proposition}
Let $c$ be an integer such that $c \geq 3$ and $c \not \equiv 1\!\pmod  3$.
Then there is exactly one Arf numerical semigroup $S$ with multiplicity $3$ and conductor $c$ given by
\begin{enumerate}[(i)]
\item $S=\langle  3, c+1, c+2 \rangle$ if $c \equiv 0\!\pmod 3$,
\item $S=\langle  3, c, c+2 \rangle$ if  $c \equiv 2\!\pmod 3$.
\end{enumerate}
\end{proposition}
\begin{proof}
\emph{(i)} If  $c \equiv 0\!\pmod 3$, then $w(1)=c + 1$ and $w(2)=c + 2$ by Lemma {\ref{2:4}}. Thus $S=\langle  3, c+1, c+2 \rangle$.

\emph{(ii)} If $c \equiv 2\!\pmod 3$, then $w(1)=c+2$ and $w(2)=c $ by Lemma {\ref{2:4}}. Hence $S=\langle  3, c, c+2 \rangle$.
\end{proof}

%Stating the above result in terms of the Frobenius number, we get
%
%
%\begin{corollary}
%Let $f$ be a positive integer which is not less than $3$ and not a multiple of $3$. Then there is exactly one Arf numerical semigroup $S$ with multiplicity $3$ and Frobenius number $f$.  We have
%
%\[S = \begin{cases} \langle  3, f+1,  f+3 \rangle & {\mbox{if }}  f \equiv 1\!\pmod 3, \\
%                               \langle  3, f+2,  f+3 \rangle &  {\mbox{if }}  f \equiv 2\!\pmod 3.
%                 \end{cases} \]
%
%\end{corollary}
%\begin{proof}
%It suffices to note that $c=f+1$.
%\end{proof}

\begin{example}
The only Arf numerical semigroup with multiplicity $3$ and Frobenius number $10$ (conductor 11) is
$\langle  3,11,13 \rangle=\{0,3,6,9,11, \rightarrow\}.$ The only Arf numerical semigroup with multiplicity $3$ and Frobenius number $11$ (conductor 12) is
$\langle  3,13,14 \rangle=\{0,3,6,9,12, \rightarrow \}.$
\end{example} 

\subsection{Arf numerical semigroups with multiplicity four}
In ~\cite{tres-cuatro}, it is shown that  numerical semigroups with multiplicity $4$ are completely determined by their genus, Frobenius number and ratio (the least minimal generator greater than the multiplicity). Formulas for the number of numerical semigroups with multiplicity $4$ and given genus and/or Frobenius number are also presented in that paper. Of course all those formulas can be expressed by using the conductor instead of the Frobenius number. Also in \cite{gen-fun} formulas for the number of numerical semigroups with multiplicity 4 and fixed Frobenius number are given; these are obtained by means of short generating functions (also if we fix the genus and the Frobenius number).

Let $S$ be an Arf numerical semigroup with multiplicity $4$ and conductor $c$. Then $c  \equiv 0, 2$ or $3 \!\pmod 4$.  The following proposition describes all Arf numerical semigroups with multiplicity $4$ and conductor $c$.

\begin{proposition} {\label{4:1}}
Let $S$ be an Arf numerical semigroup with multiplicity $4$ and conductor $c$. 
\begin{enumerate}[(i)]
\item If  $c \equiv 0\!\pmod 4$, then $S=\langle  4, 4t+2, c+1, c+3 \rangle$ for some $t\in\{1, \ldots , \frac{c}{4}\}$.
\item If  $c \equiv 2\!\pmod 4$, then $S=\langle  4, 4t+2, c+1, c+3 \rangle$ for some $t\in\{1, \ldots , \frac{c-2}{4}\}$.
\item If  $c \equiv  3\!\pmod 4$, then $S=\langle  4, c, c+2, c+3 \rangle.$
\end{enumerate}
\end{proposition}
\begin{proof}
We first note that all the semigroups given in the proposition are Arf numerical semigroups.

Let $S$ be an Arf numerical semigroup with multiplicity $4$ and conductor $c$. As we have already noted,  $c \equiv k\!\pmod m$ where $k \in \{0,2,3\}$. We have $w(3)=c-k+3$ and
\[ w(1)=\begin{cases}
c+1 & {\mbox{if }} k = 0,\\
c-k+5 & {\mbox{if }} k \neq 0,
\end{cases}\]
by Lemma {\ref{2:4}}.

\emph{(i)} If $c \equiv 0\!\pmod 4$, then $w(1)=c+1$ and $w(3)=c+3$, which by Selmer's formulas is the largest element of $\mathrm{Ap}(S,4)$. Since $w(2) < c+3$, we conclude that   $w(2)=4t+2$ with $1 \leq t \leq \frac{c}{4}$.  Thus we have
$S=\langle  4, 4t+2, c+1, c+3 \rangle, 1 \leq t \leq \frac{c}{4}$.

\emph{(ii)} If $c \equiv 2\!\pmod 4$, then $ w(3)=c+1$ and $w(1)=c+3$. In this setting, $w(1)=c+3$ is the largest element of $\mathrm{Ap}(S,4)$. Therefore $w(2) < c+3$, which implies that $w(2)=4t+2$ with $1 \leq t \leq \frac{c-2}{4}$. Thus we have $S=\langle  4, 4t+2, c+1, c+3 \rangle$, for some integer $t$ with $1 \leq t \leq \frac{c-2}{4}$.

\emph{(iii)} If $c \equiv  3\!\pmod 4$, then $ w(3)=c$ and $w(1)=c + 2$. In this case $c + 3=w(2)$ is the largest element of $\mathrm{Ap}(S,4)$. Thus $S=\langle  4, c, c+2, c+3 \rangle.$
\end{proof}

Proposition {\ref{4:1}} can be used to count Arf numerical semigroups with multiplicity $4$ and conductor $c$. Compare this result with the formula obtained for maximal embedding dimension numerical semigroups with fixed Frobenius number and genus, and multiplicity 4 given in \cite{gen-fun}.

Let $\mathrm n_A(c,m)$ denote the number of  Arf numerical semigroups with multiplicity $m$ and conductor $c$.

\begin{corollary}
Let $c$ be an integer such that $c \geq 4$ and  $c \not \equiv 1\!\pmod 4$. Then
\[\mathrm n_A(c,4)=\begin{cases}
          \frac{c}{4} & {\mbox{if }}  c \equiv 0\!\pmod 4,\\
          \frac{c-2}{4} & {\mbox{if }}  c \equiv 2\!\pmod 4,\\
           1 & {\mbox{if }}  c \equiv 3\!\pmod 4.
          \end{cases}\]
 \end{corollary}

\begin{example}
The five Arf numerical semigroups with multiplicity $4$ and conductor $20$ (Frobenius number 19) are
$$\langle  4,6,21,23 \rangle=\{0,4,6,8,10,12,14,16,18,20, \rightarrow\},$$
$$\langle  4,10,21,23 \rangle=\{0,4,8,10,12,14,16,18,20, \rightarrow\},$$
$$\langle  4,14,21,23 \rangle=\{0,4,8,12,14,16,18,20, \rightarrow\},$$
$$\langle  4,18,21,23 \rangle=\{0,4,8,12,16,18,20, \rightarrow\},$$
$$\langle  4,21,22,23 \rangle=\{0,4,8,12,16,20, \rightarrow \rangle.$$
The five Arf numerical semigroups with multiplicity $4$ and conductor $22$ (Frobenius number 21) are
$$\langle  4,6,23,25 \rangle=\{0,4,6,8,10,12,14,16,18,20,22, \rightarrow\},$$
$$\langle  4,10,23,25 \rangle=\{0,4,8,10,12,14,16,18,20,22, \rightarrow\},$$
$$\langle  4,14,23,25 \rangle=\{0,4,8,12,14,16,18,20,22, \rightarrow\},$$
$$\langle  4,18,23,25 \rangle=\{0,4,8,12,16,18,20,22, \rightarrow\},$$
$$\langle  4,22,23,25 \rangle=\{0,4,8,12,16,20,22, \rightarrow\}.$$
The only Arf numerical semigroup with multiplicity $4$ and conductor  $23$ (Frobenius number 22) is
$$\langle  4,23,25,26 \rangle=\{0,4,8,12,16,20,23, \rightarrow\}.$$
\end{example}

\subsection{Arf Numerical Semigroups with multiplicity five}

Let $S$ be an Arf numerical semigroup with multiplicity $5$ and conductor $c$. Then $c \equiv 0, 2, 3$ or $4 \!\pmod 5)$ and the following proposition describes all Arf numerical semigroups $S$ with multiplicity $5$ and conductor $c$.

\begin{proposition} {\label {5:1}}
Let $S$ be an Arf numerical semigroup $S$ with multiplicity $5$ and conductor $c$.
\begin{enumerate}[(i)]
\item If  $c \equiv 0\!\pmod 5$, then either
$S=\langle  5, c-2, c+1, c+2, c+4 \rangle$  or 
$S=\langle  5, c+1, c+2, c+3, c+4 \rangle$.
\item If  $c \equiv 2\!\pmod 5$, then
$S=\langle  5, c, c+1, c+2, c+4 \rangle$.
\item If  $c \equiv 3\!\pmod 5$, then
$S=\langle  5, c, c+1, c+3, c+4 \rangle$.
\item If  $c \equiv 4\!\pmod 5$, then either
$S=\langle  5, c-2, c, c+2, c+4 \rangle$  or 
$S=\langle  5, c, c+2, c+3, c+4 \rangle$.
\end{enumerate}
\end{proposition}
\vspace*{-0.1 cm}
\begin{proof}
We first note that all the semigroups given in the proposition are Arf numerical semigroups.

Let $S$ be an Arf numerical semigroup with multiplicity $5$ and conductor $c$. As we have already noted, $c-1$ cannot be a multiple of $5$. So, $c \equiv k\!\pmod m$ for some $k \in \{0,2,3,4\}$. Lemma \ref{2:4} asserts that $w(4)=c-k+4$ and
\[w(1)=\begin{cases}
c+1 & {\mbox{if }} k = 0,\\
c-k+6 & {\mbox{if }}  k \neq 0.
\end{cases} \]
Moreover, applying Lemma {\ref{2:5}}
\begin{equation}\label{eq:5-2}
w(4) \leq w(2)+2,
\end{equation}
and from $w(1)=w(5-4) \leq w(5-2)+(5-2)=w(3)+3$ we get
\begin{equation}\label{eq:5-3}
w(1) \leq w(3)+3.
\end{equation}
Let us also note that $w(i) \leq c+4$ for all $i\in\{1, 2, 3, 4\}$.

\emph{(i)} If $c \equiv 0\!\pmod 5$, then $ w(1)=c+1$ and $w(4)=c+4$. In light of inequality \eqref{eq:5-2}, we get $c+4=w(4) \leq w(2)+2 \leq c+6$ which implies $w(2)=c + 2$. Similarly, using \eqref{eq:5-3}, we get $c+1=w(1)\leq w(3)+3  \leq c+7$, which yields $c-2 \leq  w(3) \leq c+4$.  This implies $w(3)=c-2$ or $w(3)=c+3$. It follows that
\[S=\langle  5,  c-2,  c+1,  c+2,  c+4 \rangle \ {\mbox{or}} \
 S=\langle  5,  c+1,  c+2,  c+3,  c+4 \rangle.\]
 
\emph{(ii)} If $c \equiv 2\!\pmod 5$, then $w(1)=c+4$ and $w(4)=c+2$. Using inequality \eqref{eq:5-2},  we get
$c+2 = w(4) \leq w(2) + 2 \leq c+6$,  which gives $c \leq w(2)\leq c+4$. Consequently, $w(2)=c$. Analogously, from  \eqref{eq:5-3}, we get
$c+4 = w(1) \leq w(3)+3 \leq c+7$, which yields $ c+1 \leq w(3) \leq c+4$
and this implies $w(3)= c+1$. It follows that
\[S=\langle  5,  c,  c+1,  c+2,  c+4 \rangle.\]

\emph{(iii)} If $c \equiv 3\!\pmod 5$, then $w(1)=c+3$ and $w(4)=c+1$. In this case, $c+4=w(2)$ is the largest element of $\mathrm{Ap}(S,5)$. As before, \eqref{eq:5-3} yields
$c+3=w(1)  \leq w(3)+3  \leq c+7$, and thus $c \leq w(3) \leq c+4$. Hence $w(3)=c$, and 
\[S=\langle  5,  c,  c+1,  c+3,  c+4 \rangle.\]

\emph{(iv)} If $c \equiv 4\!\pmod 5$, then $ w(1)=c+2$ and $w(4)=c$. In this case, $c+4=w(3)$ is the largest element of $\mathrm{Ap}(S,5)$.  Applying \eqref{eq:5-2}, we get
$c= w(4) \leq w(2)+2 \leq c+6$, whence $c-2 \leq w(2) \leq c+4$. 
Thus $w(2)=c - 2$ or $w(2)=c+3$. It follows that
\[S=\langle  5,  c-2,  c,  c+2,  c+4 \rangle \ {\mbox{or}} \
 S=\langle  5,  c,  c+2,  c+3,  c+4 \rangle.\qedhere \]
\end{proof} 

As a consequence of Proposition {\ref{5:1}}, we can count the number of Arf numerical semigroups with multiplicity $5$ and conductor $c$, with $c$ an integer greater than or equal to five and $c \not \equiv 1\!\pmod 5$.

\begin{corollary}
Let $c$ be an integer such that $c \geq 5$ and $c \not \equiv 1\!\pmod 5)$. Then
\[\mathrm n_A(c,5)= \begin{cases}
2 & {\mbox{if }}  c \equiv 0\!\pmod 5 \  {\mbox{or}} \ c \equiv 4\!\pmod 5,\\
1 & {\mbox{if }}  c \equiv 2\!\pmod 5 \  {\mbox{or}} \ c \equiv 3\!\pmod 5.
\end{cases}\]
\end{corollary} 

\begin{example}
The two Arf numerical semigroups with multiplicity $5$ and conductor $30$ (Frobenius number 29) are
$\langle  5,28,31,32,34 \rangle=\{0,5,10,15,20,25,28,30, \rightarrow\}$ and
$ \langle  5,31,32,33,34 \rangle=\{0,5,10,15,20,25,30 \rightarrow\}.$ 
\begin{verbatim}
gap> l:=NumericalSemigroupsWithFrobeniusNumber(29);;
gap> l5:=Filtered(l,s->MultiplicityOfNumericalSemigroup(s)=5);;
gap> Filtered(l5,IsArfNumericalSemigroup);
[ <Numerical semigroup with 5 generators>, 
  <Numerical semigroup with 5 generators> ]
gap> List(last,MinimalGenerators);
[ [ 5, 28, 31, 32, 34 ], [ 5, 31, 32, 33, 34 ] ]
\end{verbatim}

The only Arf numerical semigroup with multiplicity $5$ and conductor $32$ (Frobenius number 31) is
$\langle  5,32,33,34,36 \rangle=\{0,5,10,15,20,25,30,32, \rightarrow\}.$ 

The only Arf numerical semigroup with multiplicity $5$ and conductor $33$ (Frobenius number 32) is
$\langle  5,33,34,36,37 \rangle=\{0,5,10,15,20,25,30,33 \rightarrow\}.$ 

The two Arf numerical semigroups with multiplicity  $5$ and conductor $34$ (Frobenius number $33$) are
\[
\begin{array}{l}
\langle  5,32,34,36,38 \rangle=\{0,5,10,15,20,25,30,32,34, \rightarrow\},\\ 
\langle  5,34,36,37,38 \rangle=\{0,5,10,15,20,25,30,34, \rightarrow\}.
\end{array}
\]
\end{example}

\subsection{Arf Numerical Semigroups with multiplicity six}

To determine all Arf numerical semigroups with multiplicity $6$ and a given conductor $c$, we will make use of the ratio of a numerical semigroup. Recall that for a given numerical semigroup $S$, the \emph{ratio} is the smallest integer in $S$ that is not a multiple of its multiplicity, or in other words, the smallest minimal generator greater than the multiplicity \cite{tres-cuatro}. We will use $r$ to denote the ratio of $S$.
 
Let $S$ be an Arf numerical semigroup with multiplicity $6$ and conductor $c$. Then $c \equiv 0, 2, 3, 4$ or $5 \!\pmod 6)$. 

Since $c+5$ is the largest element of the minimal set of generators of $S$, as the second least element of the minimal set of generators for $S$, the ratio of $S$ satisfies
\begin{equation*} %\label{eq:r-c}
r \leq c +1.
\end{equation*}
The following proposition describes all Arf numerical semigroups $S$ with multiplicity $6$ and conductor $c$.

\begin{proposition} {\label{6:1}}
Let $S$ be an Arf numerical semigroup with multiplicity $6$ and conductor $c$.
\begin{enumerate}[(i)]
\item If  $c \equiv 0\!\pmod 6$, then $S$ equals one of the following numerical semigroups
\[\begin{array}{l}
\langle  6, c+1, c+2, c+3, c+4, c+5 \rangle,\\
\langle  6, 6u+2, 6u+4, c+1, c+3, c+5 \rangle,\\
\langle  6, 6u+3, c+1, c+2, c+4, c+5 \rangle, \\
\langle  6, 6u+4, 6u+8, c+1, c+3, c+5 \rangle,
\end{array}\]
for some integer $u$ such that  $1 \leq u \leq \frac{c}{6} - 1$.

\item If  $c \equiv 2\!\pmod 6$, then $S$ is of one of the following forms
\[\begin{array} {l}
 \langle  6, 6t+2, 6t+4, c+1, c+3, c+5 \rangle,\\
 \langle  6, 6u+3, c, c+2, c+3, c+5 \rangle ,\\
 \langle  6, 6u+4, 6u+8, c+1, c+3, c+5 \rangle,
\end{array}\]
for some integers $t$ and $u$ with $1 \leq t \leq \frac{c - 2}{6} $ and $1 \leq u \leq \frac{c - 2}{6} - 1$.

\item If  $c \equiv 3\!\pmod 6$, then
\[S=\langle  6, 6t+3, c+1, c+2, c+4, c+5 \rangle,\]
for some integer $t$ such that $1 \leq t \leq \frac{c - 3}{6}$.

\item If  $c \equiv 4\!\pmod 6$, then $S$ is equal to one of the following numerical semigroups
\[\begin{array} {l}
\langle  6, 6t+2, 6t+4, c +1, c+3, c+5 \rangle,\\
\langle  6, 6t+4, 6t+8, c +1, c+3, c+5 \rangle,
\end{array}\]
for some integer $t$ with $1 \leq t \leq \frac{c - 4}{6}$.

\item If $c \equiv 5\!\pmod 6)$, then $S$ is of one of the following forms
\[\begin{array} {l}

\langle  6, c, c+2, c+3, c+4, c+5 \rangle,\\
\langle  6, 6t+3, c, c+2, c+3, c+5 \rangle,
\end{array}\]
for some integer $t$ with $1 \leq t \leq \frac{c - 5}{6}$.
\end{enumerate}
\end{proposition}
\begin{proof}
We first note that all the semigroups given in the proposition are Arf numerical semigroups.

Let $S$ be an Arf numerical semigroup with multiplicity $6$ and conductor $c$. As we have already mentioned, $c \equiv k\!\pmod 6$ for some $k \in \{0,2,3,4,5\}$. By Lemma \ref{2:4}, we have $w(5)=c-k+5$ and
\begin{equation}\label{eq:6-1}
w(1)= \begin{cases}
c+1 & {\mbox{if }} k = 0,\\
c-k+7 & {\mbox{if }}  k \neq 0.
\end{cases} 
\end{equation} 
Also, $ w(4) \leq w(2)+2$  and $w(2)=w(6-4) \leq w(6-2)+(6-2)=w(4)+4$ by Lemma {\ref{2:5}}. Combining these two inequalities we get
\begin{equation}\label{eq:6-2}
w(4)-2 \leq w(2) \leq w(4)+4.
\end{equation}
Let us also note that $w(i) \leq c+5$ for all $i\in\{1, 2, 3, 4, 5\}$.

\emph{(i)} If $c \equiv 0\!\pmod 6$, then $ w(1)=c+1$ and $w(5)=c+5$ by \eqref{eq:6-1}.  The ratio
$r$ of $S$ is one of $w(1)=c+1$, $w(2)$, $w(3)$ or $w(4)$.
\begin{enumerate}[--]
\item If $r = c+1$, then $w(2)= c+2$, $w(3)= c+3$, $w(4)= c+4$, and consequently
\[S=\langle  6, c+1, c+2, c+3, c+4, c+5 \rangle.\]
\item If $r = w(2)$, then  $w(2)<c$, $w(2) < w(3)$ and $w(2)<w(4)$ by the definition of ratio. Hence $w(3)=c+3$ by Lemma {\ref{2:3}} and $w(2) = w(4)-2$ or equivalently, $w(4)=w(2)+2$ by \eqref{eq:6-2}. Write $w(2)=6u+2$. Then $w(4)=6u+4$ and
\[S=\langle  6, 6u+2, 6u+4, c+1, c+3, c+5 \rangle,\]
with $1 \leq u \leq \frac{c}{6}-1$.
\item If $r = w(3)$, then $w(3)<c$. Also $w(3) < w(2)$ and $w(3) < w(4)$ by the definition of ratio. Hence $w(2)= c+2$ and $w(4)= c+4$ by Lemma {\ref{2:3}}. By setting $w(3)=6u+3$, we get
\[S=\langle  6, 6u+3, c+1, c+2, c+4, c+5 \rangle,\]
with $1 \leq u \leq \frac{c}{6}-1.$
\item  If $r = w(4)$, then  the definition of ratio forces $w(4)<c$, $w(4) < w(3)$ and $w(4) < w(2)$. Hence  $w(3)= c+3$  by Lemma {\ref{2:3}} and $w(2) = w(4)+4$ by \eqref{eq:6-2}. Put $w(4)=6u+4$. Then $w(2)=6u+8$ and we have
\[S=\langle  6, 6u+4, 6u+8, c+1, c+3, c+5 \rangle,\]
with $1 \leq u \leq \frac{c}{6}-1.$

\end{enumerate}

\emph{(ii)} If $c \equiv 2\!\pmod 6$, then $ w(1)=c+5$ and $w(5)=c+3$ by \eqref{eq:6-1}. Also we have in this setting that $r\in\{ w(2), w(3), w(4)\}$.
\begin{enumerate}[--]
\item If $r = w(2)$, the definition of ratio yields $w(2) \leq c$, $w(2) < w(3)$ and $w(2) < w(4)$. Hence  $w(3)= c+1$ by Lemma {\ref{2:3}}, and $w(2) = w(4)-2$, or equivalently, $w(4)=w(2)+2$ by \eqref{eq:6-2}.  Write $w(2)=6t+2$. Then  $w(4)=6t+4$  and we obtain
\[S=\langle  6, 6t+2, 6t+4, c+1, c+3, c+5 \rangle,\]
with $1 \leq t \leq \frac{c-2}{6}$.
\item If $r = w(3)$, then $w(3)< w(2)$ and $w(3) < w(4)$ by the definition of the ratio. Hence $w(4)= c+2$  and $w(2)=c$ by Lemma {\ref{2:3}}. Note also that $w(3) \leq c-5$, since $w(3)<c$ and $w(3)\equiv c+1 \!\pmod 6$. By denoting $w(3)=6u+3$, we get
\[S=\langle  6, 6u+3, c, c+2, c+3, c+5 \rangle,\]
where $1 \leq u \leq \frac{c-2}{6}-1$.
\item If $r = w(4)$, then $w(4)<c$, $w(4)<w(2)$ and $w(4)<w(3)$ by the definition of ratio. Hence $w(3)= c+1$ by Lemma {\ref{2:3}}, and $w(2) = w(4)+4$ by \eqref{eq:6-2}. Put $w(4)=6u+4$. Then $w(2) = 6u+8$, and we have
\[S=\langle  6, 6u+4, 6u+8, c+1, c+3, c+5 \rangle,\]
with $1 \leq u \leq \frac{c-2}{6}-1$.
\end{enumerate}

\emph{(iii)} If $c \equiv 3\!\pmod 6$, then $ w(1)=c+4$ and $w(5)=c+2$ by \eqref{eq:6-1}. In this case, $c+5=w(2)$ is the largest element in $\mathrm{Ap}(S,6)$. Using \eqref{eq:6-2},
$c+5 =w(2) \leq w(4)+4$ which implies  $c+1 \leq w(4)$ and thus $w(4)=c+1$. Therefore, the only possibility for the ratio is $r = w(3)$ and if we express $w(3)=6t+3$, we get
\[S=\langle  6, 6t+3, c+1, c+2, c+4, c+5 \rangle,\]
where $1 \leq t \leq \frac{c-3}{6}$.

\emph{(iv)} If $c \equiv 4\!\pmod 6$, then $ w(1)=c+3$ and $w(5)=c+1$ by \eqref{eq:6-1}. In this case, $c+5 =w(3)$ is the largest element of $\mathrm{Ap}(S,6)$. Since $w(4) \leq c$, the ratio $r$ of $S$ is either $w(2)$ or $w(4)$.
\begin{enumerate}[--]
\item If $r = w(2)$, then $w(2)< w(4)$ and thus $w(4)=w(2)+2$ by  \eqref{eq:6-2}. Write $w(2)$ as $w(2)=6t+2$. Then $w(4)=6t+4$ and
\[S=\langle  6, 6t+2, 6t+4, c+1, c+3, c+5 \rangle,\]
with $1 \leq t \leq \frac{c-4}{6}$.
\item If $r = w(4)$, then $w(4) < w(2)$ and thus $w(2)=w(4)+4$ by  $(6.2)$.  Put  $w(4)=6t+4$. Then $w(2)=6t+8$ and
$$S=\langle  6, 6t+4, 6t+8, c+1, c+3, c+5 \rangle$$
\noindent where $1 \leq t \leq \frac{c-4}{6}.$
\end{enumerate}
\emph{(v)} If $c \equiv 5\!\pmod 6$, then $ w(1)=c+2$ and $w(5)=c$ by \eqref{eq:6-1}. In this case, $c+5=w(4)$ is the largest element of $\mathrm{Ap}(S,6)$. Using  \eqref{eq:6-2}, we obtain $c+5 =w(4) \leq w(2)+2$ and then $c+3 \leq w(2)$. This implies $w(2)=c+3$. Therefore, either $r = w(5) = c$ or $r = w(3)$. 
\begin{enumerate}[--]
\item If $r = c$, then 
\[S=\langle  6, c, c+2, c+3, c+4, c+5 \rangle.\]
\item  If $r = w(3)$ and if we put $w(3)=6t+3$, then 
\[S=\langle  6,  6t+3,  c,  c+2, c+3, c+5 \rangle,\]
where $1 \leq t \leq \frac{c-5}{6}$.\qedhere
\end{enumerate}
\end{proof}

If $c$ is an integer such that $c \geq 6$ and $c \not \equiv 1\!\pmod 6)$, then Proposition {\ref{6:1}} can be used to count Arf numerical semigroups with multiplicity $6$ and conductor $c$.

\begin{corollary}
Let $c$ be an integer  such that $c \geq 6$  and $c \not \equiv 1\!\pmod 6)$. Then
\[\mathrm n_A(c,6)= \begin{cases}
          \frac{c}{2}-2 & {\mbox{if }}  c \equiv 0\!\pmod 6, \\
          \frac{c-2}{2}-2 & {\mbox{if }}  c \equiv 2\!\pmod 6, \\
          \frac{c-3}{6} & {\mbox{if }}  c \equiv 3\!\pmod 6,\\
          \frac{c-4}{3} & {\mbox{if }}  c \equiv 4\!\pmod 6, \\
          \frac{c+1}{6} & {\mbox{if }}  c \equiv 5\!\pmod 6.
\end{cases}
\]
\end{corollary}

\begin{example}
The 13 Arf numerical semigroups with multiplicity $6$ and conductor $30$ (Frobenius number 29) are
$$ \langle  6,31,32,33,34,35 \rangle=\{0,6,12,18,24,30, \rightarrow\},$$
$$\langle  6,8,10,31,33,35 \rangle=\{0,6,8,10,12,14,16,18,20,22,24,26,28,30, \rightarrow\},$$
$$ \langle  6,14,16,31,33,35 \rangle=\{0,6,12,14,16,18,20,22,24,26,28,30, \rightarrow\},$$
$$ \langle  6,20,22,31,33,35 \rangle=\{0,6,12,18,20,22,24,26,28,30, \rightarrow\},$$
$$ \langle  6,26,28,31,33,35 \rangle=\{0,6,12,18,24,26,28,30, \rightarrow\},$$
$$\langle  6,9,31,32,34,35 \rangle=\{0,6,9,12,15,18,21,24,27,30, \rightarrow\},$$
$$\langle  6,15,31,32,34,35 \rangle=\{0,6,12,15,18,21,24,27,30, \rightarrow\},$$
$$\langle  6,21,31,32,34,35 \rangle=\{0,6,12,18,21,24,27,30, \rightarrow\},$$
$$\langle  6,27,31,32,34,35 \rangle=\{0,6,12,18,24,27,30, \rightarrow\},$$
$$\langle  6,10,14,31,33,35 \rangle=\{0,6,10,12,14,16,18,20,22,24,26,28,30, \rightarrow\},$$
$$\langle  6,16,20,31,33,35 \rangle=\{0,6,12,16,18,20,22,24,26,28,30, \rightarrow\},$$
$$\langle  6,22,26,31,33,35 \rangle=\{0,6,12,18,22,24,26,28,30, \rightarrow\},$$
$$\langle  6,28,32,31,33,35 \rangle=\{0,6,12,18,24,28,30, \rightarrow\}.$$
\end{example}

\end{document}